\documentclass[a4paper]{amsart}

\usepackage{lmodern,microtype}
\usepackage{amsmath,amstext,amssymb,mathrsfs,amscd,amsthm,indentfirst}
\usepackage{amsfonts}
\usepackage[dvipsnames,svgnames,x11names,hyperref]{xcolor}

\usepackage[pagebackref,colorlinks,citecolor=Mahogany,linkcolor=Mahogany,urlcolor=Mahogany,filecolor=Mahogany]{hyperref}

\usepackage{enumerate}
\usepackage{graphicx}
\usepackage{multicol}
\usepackage[utf8]{inputenc}
\usepackage{mathtools}
\usepackage{enumitem}

\usepackage{tikz, tikz-cd}
\usetikzlibrary{matrix,calc,positioning,arrows,decorations.pathreplacing,decorations.markings,patterns} 
\tikzset{->-/.style={decoration={
			markings,
			mark=at position #1 with {\arrow{>}}},postaction={decorate}}} 
\tikzset{-<-/.style={decoration={
					markings,
					mark=at position #1 with {\arrow{<}}},postaction={decorate}}}

\newcommand{\C}{\mathbb{C}}

\newcommand{\cS}{\mathcal{S}}
\newcommand{\cT}{\mathcal{T}}
\newcommand{\surf}{\mathcal{M}}
\newcommand{\del}{\partial}
\newcommand{\msurf}{\mathring{\surf}} 
\newcommand{\mcT}{\mathring{\cT}} 
\newcommand{\cE}{\mathcal{E}} 
\newcommand{\cF}{\mathcal{F}} 
\newcommand{\bF}{\mathbb{F}} 
\newcommand{\mI}{\mathring{I}} 
\newcommand{\mR}{\mathring{R}} 
\newcommand{\md}{\mathring{d}} 
\newcommand{\me}{\mathring{e}} 
\newcommand{\delout}{\del^{\mathrm{out}}} 

\newcommand{\scU}{\mathscr{U}} 
\newcommand{\scV}{\mathscr{V}} 
\newcommand{\mscV}{\mathring{\scV}} 
\newcommand{\scX}{\mathscr{X}} 
\newcommand{\scY}{\mathscr{Y}} 

\newcommand{\Mon}{\mathbb{M}} 
\newcommand{\uw}{\underline{w}} 
\newcommand{\ba}{\mathbf{a}} 
\newcommand{\bb}{\mathbf{b}} 
\newcommand{\bd}{\mathbf{d}} 
\newcommand{\bx}{\mathbf{x}} 
\newcommand{\by}{\mathbf{y}} 
\newcommand{\bp}{\mathbf{p}} 
\newcommand{\bq}{\mathbf{q}} 
\newcommand{\cont}{\mathfrak{c}} 
\newcommand{\Cont}{\mathfrak{C}} 

\DeclareMathOperator{\Diff}{Diff}
\newcommand{\Diffd}{\Diff_\del}

\DeclareMathOperator{\Q}{\mathbb{Q}}
\DeclareMathOperator{\Z}{\mathbb{Z}}

\mathchardef\ordinarycolon\mathcode`\:
\mathcode`\:=\string"8000
\begingroup \catcode`\:=\active
  \gdef:{\mathrel{\mathop\ordinarycolon}}
\endgroup

\newcommand{\pa}[1]{\left(#1\right)}

\newcommand{\set}[1]{\left\{#1\right\}}

\usepackage{amsthm}
\theoremstyle{plain}
\newtheorem{theorem}{Theorem}[section]
\newtheorem{proposition}[theorem]{Proposition}

\newtheorem{lemma}[theorem]{Lemma}
\newtheorem{corollary}[theorem]{Corollary}

    \newtheorem{atheorem}{Theorem}
    
    \newtheorem{aquestion}{Question}
    
    \newtheorem{aconjecture}{Conjecture}

\theoremstyle{definition}
\newtheorem{definition}[theorem]{Definition}

\newtheorem{notation}[theorem]{Notation}

\theoremstyle{remark}
\numberwithin{equation}{section}

\title[Non-trivial action of the Johnson filtration on homology]{Non-trivial action of the Johnson filtration on the homology of configuration spaces}

\author{Andrea Bianchi}
\thanks{
A.B. was supported by the
\emph{European Research Council} under the
European Union’s Horizon 2020 research and innovation programme (grant agreement No. \texttt{772960}),
and by the
\emph{Danish National Research Foundation} through the \emph{Copenhagen Centre for
Geometry and Topology} (\texttt{DNRF151}). A.S. was funded by a studentship of the \emph{Engineering and Physical Sciences Research Council} (project reference: 2261124).
}
\email{anbi@math.ku.dk}
\address{Mathematical Institute, University of Copenhagen, Universitetsparken 5, Copenhagen, DK}

\author{Andreas Stavrou}
\email{as2558@cam.ac.uk}
\address{Centre for Mathematical Sciences, University of Cambridge, Wilberforce Road, Cambridge CB3 0WB, UK}

\keywords{Configuration spaces, Johnson's filtration}
\subjclass[2020]{
55R80, 
}

\date{\today}

\begin{document}
\maketitle

\begin{abstract}
    We let the mapping class group $\Gamma_{g,1}$ of a genus $g$ surface $\Sigma_{g,1}$ with one boundary component act on the homology
    $H_*(F_{n}(\Sigma_{g,1});\Q)$ of the $n$\textsuperscript{th} ordered configuration space of the surface. We prove that the action is non-trivial when restricted to the $(n-1)$\textsuperscript{st} stage of the Johnson filtration, for all $n\ge 1$ and $g\ge 2$. We deduce an analogous result for closed surfaces.
\end{abstract}

\section{Introduction}
\subsection{Background and statement of results}
Let $Z$ be a topological space. For $n\ge1$, we denote by $F_n(Z)$ the $n$\textsuperscript{th} ordered configuration space
\[
F_n(Z)=\set{(z_1,\dots,z_n)\in Z^n\,|\,z_i\neq z_j\mbox{ for }i\neq j}.
\]
Let $g\ge2$ and let $\surf=\Sigma_{g,1}$ denote a compact orientable surface of genus $g$ with one boundary curve. We are interested in the configuration spaces $F_n(\surf)$, and in particular in their homology groups, regarded as representations of the mapping class group $\Gamma_{g,1}=\pi_0(\Diffd(\surf))$. More concretely, we are interested in the following question.
\begin{aquestion}
\label{que:main}
What is the kernel of the action of $\Gamma_{g,1}$ on $H_*(F_n(\surf))$? In other words, which mapping classes in $\Gamma_{g,1}$ act trivially on $H_*(F_n(\surf))$?
\end{aquestion}
The Johnson filtration of the mapping class group
\[
\Gamma_{g,1}= J_{g,1}(0)\supseteq J_{g,1}(1)\supseteq J_{g,1}(2)\supseteq J_{g,1}(3)\dots
\]
seems to play a key role in answering Question \ref{que:main}: for a fixed basepoint $*\in\partial\surf$,
Moriyama \cite{Moriyama} proves that $J_{g,1}(n)$ is precisely the kernel of the action of $\Gamma_{g,1}$ on the homology of the quotient of $F_n(\surf)$ by the subspace of configurations $(z_1,\dots,z_n)$ satisfying $z_i=*$ for some $1\le i\le n$. Leveraging on this result, the first author, Miller and Wilson \cite{bianchi2021mapping} prove that $J_{g,1}(n)$ is contained in the kernel of the action of $\Gamma_{g,1}$ on $H_*(F_n(\surf))$, for all $n\ge0$, and make the following conjecture.
\begin{aconjecture}
\label{conj:main}
The kernel of the action of $\Gamma_{g,1}$ on $H_*(F_n(\surf))$ is the subgroup of $\Gamma_{g,1}$ generated by $J_{g,1}(n)$ and the boundary Dehn twist $T_{\partial\surf}$.
\end{aconjecture}
The first main result of this article is a step towards a proof of Conjecture \ref{conj:main}.
\begin{atheorem}
\label{thm:main}
For $g\ge2$ and $n\ge1$, the $(n-1)$\textsuperscript{st} stage of the Johson filtration $J_{g,1}(n-1)$ acts non-trivially on $H_n(F_n(\surf))$.
\end{atheorem}
We remark that the statement of Theorem \ref{thm:main} fails for $g=1$: the subgroup $J_{1,1}(2)\subset\Gamma_{1,1}$ is the cyclic subgroup generated by $T_{\partial\surf}$, and thus acts trivially on $H_*(F_n(\Sigma_{1,1}))$ for all $n\ge1$.

In order to prove Theorem \ref{thm:main}, we construct a triple $(\phi, x,y)$ of a mapping class $\phi\in J_{g,1}(n-1)$,
a homology class $x\in H_n(F_n(\surf))$ and a cohomology class $y\in H^n(F_n(\surf))$, and prove that the homology class $\phi_*(x)-x$ is non-trivial by proving that the Kronecker pairing $\left< \phi_*(x)-x,y\right>\in\Z$ does not vanish. The same method has been used in \cite{Bianchi2020,bianchi2021mapping} to prove the statement of Theorem \ref{thm:main} for $n\le3$. In order to treat the general case, we develop a calculus for higher \textit{contents} of elements of free groups; roughly speaking, a content of a word $\uw$ in a free group counts the number of subwords of a certain type; a similar method has been used by the second author in \cite{stavrou2021} in the study of the rational homology of unordered configuration spaces of surfaces. After the construction of $(\phi, x,y)$, we identify the relevant Kronecker pairings with certain contents, and then we compute these contents.

We then shift our focus to the following variant of Question \ref{que:main} for closed surfaces.
\begin{aquestion}
\label{que:main2}
Let $g\ge2$ and $n\ge1$, let $\Sigma_g$ denote a closed, orientable surface of genus $g$. What is the kernel of the action of the mapping class group $\Gamma_g=\pi_0(\Diff^+(\Sigma_g))$ on $H_*(F_n(\Sigma_g))$?
\end{aquestion}
Let us define $J_g(n)\subset\Gamma_g$ as the image of $J_{g,1}(n)$ along the natural, surjective group homomorphism $\Gamma_{g,1}\to\Gamma_g$.
Looijenga \cite{Looijenga} has proved that the Torelli group $J_g(1)\subset\Gamma_g$ acts non-trivially on $H_*(F_3(\Sigma_g))$ for $g\ge3$. On the other hand, the arguments of \cite{bianchi2021mapping} can be adapted to show that for $g\ge1$ and $n\ge1$ there is a $\Gamma_{g,1}$-equivariant, finite chain complex computing the homology of $F_n(\Sigma_g)$ as a $\Gamma_{g,1}$-representation, and such that $J_{g,1}(n)$ acts trivially on this complex: hence $J_g(n)$ acts trivially on $H_*(F_n(\Sigma_g))$.

As a corollary of Theorem \ref{thm:main} we will prove the following, which is the second main result of the paper.
\begin{atheorem}
\label{thm:main2}
For $g\ge2$ and $n\ge1$, the subgroup $J_{g+1}(n-1)\subset\Gamma_{g+1}$ acts non-trivially on $H_{n+1}(F_{n+1}(\Sigma_{g+1}))$. 
\end{atheorem}
We remark again that the statement of Corollary \ref{thm:main2} fails for $g+1=1$, since we have that $J_1(n-1)$ is the trivial group for $n\ge2$.

Similarly, for $g+1=2$, we have the equality $J_2(2)=J_2(1)$, as both groups are generated by separating Dehn twists. We are not able to prove non-triviality of the action of the Johnson filtration of $\Gamma_2$ on $H_*(F_n(\Sigma_2))$, and we propose the following conjecture.
\begin{aconjecture}
\label{conj:Gamma2symplectic}
The Torelli group $J_2(1)\subset\Gamma_2$ acts trivially on $H_*(F_n(\Sigma_2))$ for all $n\ge1$; in other words, $H_*(F_n(\Sigma_2))$ are symplectic representations of $\Gamma_2$.
\end{aconjecture}
It was communicated to us that Dan Petersen and Orsola Tommasi have a proof of this conjecture for homology with rational coefficients \cite{PT}.

\subsection{Guide to reader}
The paper is written so as to unambiguously construct $x,y$ and $\phi$ in the body of Section 3, however a reader happy to stare at pictures is invited to only consult the figures of that section. Furthermore, if the reader believes that the Kronecker pairing is equivalent to submanifold intersections and that this equivalence is compatible with products of submanifolds and restrictions to smaller ambient manifold, then they are invited to skip section 4 until Definition \ref{defn:contents} and conclude for themselves why the sum of the terms in Proposition \ref{prop:reductiontocontents} is equal to $\left< \phi_*(x)-x,y\right>\in\Z$.
The main computation is in Section 5.

\subsection{Acknowledgments}
The first author would like to thank Dan Petersen, Orsola Tommasi, Jeremy Miller and Jennifer Wilson for the useful conversations. The second author would like to thank his PhD supervisor Oscar Randal-Williams for introducing him to this problem and the work of the first author and all the support. He would also like to thank the Copenhagen Centre for Geometry and Topology for its hospitality during visits where the paper was started and finished.
Both authors would like to thank Louis Hainaut for comments on a first draft of the article.

\tableofcontents
\section{Preliminaries}
In the entire article, unless stated otherwise, homology and cohomology are taken with coefficients in $\mathbb{Z}$; in fact, all arguments work for homology and cohomology over any commutative ring $R$.
\subsection{The Johnson filtration}
Let $*\in\partial\surf$ be a fixed basepoint, and let $\pi=\pi_1(\surf,*)$. The lower central series of $\pi$, denoted
\[
\pi=\gamma_0\pi\supset\gamma_1\pi\supset\gamma_2\pi\supset\dots,
\]
is defined recursively by $\gamma_0\pi=\pi$ and $\gamma_{i+1}\pi=[\pi,\gamma_i\pi]$, for $i\ge0$: here $[\pi,\gamma_i\pi]$ is the subgroup of $\pi$ generated by all commutators of an element of $\pi$ and an element of $\gamma_i\pi$.

All subgroups $\gamma_i\pi\subset\pi$ are characteristic, in particular any automorphism of $\pi$ restricts to an automorphism of $\gamma_i\pi$ and induces an automorphism of the quotient group $\pi/\gamma_i\pi$. In particular, the natural action of $\Gamma_{g,1}$ on $\pi$ induces an action on $\pi/\gamma_i\pi$ by group automorphism
\begin{definition}
\label{defn:Johnsonfiltration}
For $i\ge0$ we denote by $J_{g,1}(i)\subset\Gamma_{g,1}$ the kernel of the action of $\Gamma_{g,1}$ on $\pi/\gamma_i\pi$, i.e. the subgroup of mapping classes acting trivially on $\pi/\gamma_i\pi$.
\end{definition}
For example, for $i=1$, we have that $J_{g,1}(1)$ is the Torelli group, i.e. the subgroup of $\Gamma_{g,1}$ of mapping classes acting trivially on $\pi/[\pi,\pi]\cong H_1(\surf)$. A typical example of a mapping class in $J_{g,1}(1)$ is a \emph{bounding pair} (see Figure \ref{fig:boundingpairs}): if $\alpha,\alpha'$ are disjoint, non-isotopic, non-separating simple closed curves in $\surf$ that together bound a subsurface of $\surf$, then the bounding pair $BP_{\alpha,\alpha'}$ is defined as the product $D_\alpha D_{\alpha}^{-1}\in\Gamma_{g,1}$ of the Dehn twist around $\alpha$ and the inverse of the Dehn twist around $\alpha'$. It is known that bounding pairs belong to $J_{g,1}(1)$, and in fact $J_{g,1}(1)$ is generated by bounding pairs and by Dehn twists around separating simple closed curves in $\surf$ \cite{Birman,Powell} (see also  \cite[Corollary 1.4]{Putman}).

In this article, to obtain elements in $J_{g,1}(n)$ for $n\ge2$ we will use that the Johnson filtration of $\Gamma_{g,1}$ is a \emph{central filtration}: for all $i,j\ge0$, the subgroup $[J_{g,1}(i),J_{g,1}(j)]\subset\Gamma_{g,1}$, generated by commutators of an element in $J_{g,1}(i)$ and an element in $J_{g,1}(j)$, is contained in $J_{g,1}(i+j)$ \cite{Johnson} (see also \cite[Proposition 2.9]{ChurchPutman}).

\subsection{Action of diffeomorphisms and Poincare duality}
\label{subsec:Poincare}
We denote by $\msurf=\surf\setminus\del\surf$ the interior of $\surf$.
The space $F_n(\surf)$ is a non-compact manifold with boundary\footnote{The interior of this manifold has a natural smooth structure, whereas the boundary has a Whitney stratification by smooth submanifolds, but not a canonical smooth structure.} of dimension $2n$; a configuration $(z_1,\dots,z_n)$ lies in $\del F_n(\msurf)$ if and only if at least one $z_i$ lies on $\del\surf$.
Thus $F_n(\msurf)$ is the interior of $F_n(\surf)$, and hence the inclusion $F_n(\msurf)\subset F_n(\surf)$ is a homotopy equivalence. We will henceforth focus on $F_n(\msurf)$, which is an orientable $2n$-dimensional manifold without boundary, equipped with a canonical smooth structure. A once and for all fixed orientation on $\surf$ induces a canonical orientation on $\msurf^m$ and on each open subspace thereof.

The group of diffeomorphisms $\Diffd(\surf)$ acts naturally on $F_n(\msurf)$ by diffeomorphisms, and hence it acts on the homology groups $H_*(F_n(\msurf))$; isotopic diffeomorphisms of $\surf$ induce isotopic diffeomorphisms of $F_n(\msurf))$, and hence the same map in homology. This gives rise to an action of $\Gamma_{g,1}$ on $H_*(F_n(\msurf))$. Similarly, there is an action of $\Gamma_{g,1}$ on the cohomology of $F_n(\msurf)$. The Kronecker pairing is balanced with respect to these actions: for all $x\in H_i(F_n(\msurf))$,  $y\in H^i(F_n(\msurf))$ and $\phi\in \Gamma_{g,1}$ we have the equality
\[
\left<\phi_*(x),y\right>=\left<x,\phi^*(y)\right>;
\]
in particular we have $\left<\phi_*(x)-x,y\right>=\left<x,\phi^*(y)-y\right>$.

We can also consider the action of $\Diffd(\surf)$ on the cartesian power $\surf^{n}$; this action preserves the following subspaces of $\surf^n$:\footnote{The notation is taken from \cite{bianchi2021mapping}, and inspired by the notation of \cite{Moriyama}.}
\begin{itemize}
    \item $A'_n(\surf)=\set{(z_1,\dots,z_n)\in\surf^n\,|\,z_i\in\del\surf\mbox{ for some } i}$;
    \item $\Delta_n(\surf)=\set{(z_1,\dots,z_n)\in\surf^n\,|\, z_i=z_j \mbox{ for some } i\neq j}$.
\end{itemize}
We have therefore an induced action of $\Diffd(\surf)$, descending to the mapping class group $\Gamma_{g,1}$ on the relative homology $H_*(\surf^n,\Delta_n(\surf)\cup A'_n(\surf))$.

Poincare duality for the orientable $2n$-manifold $F_n(\msurf)$, together with the identification of the Borel-Moore homology of $F_n(\msurf)$ with the relative homology of the pair $(\surf^n,\Delta_n(\surf)\cup A'_n(\surf))$, gives an isomorphism of $\Gamma_{g,1}$-representations
\[
H^{*}(F_n(\msurf))\cong H_{2n-*}(\surf^n,\Delta_n(\surf)\cup A'_n(\surf)).
\]
We are going to use also the following simple principles:
\begin{itemize}
    \item Let $N_1\subset F_n(\msurf)$ be an oriented, compact $i$-submanifold, and let $N_2\subset F_n(\msurf)$ be an oriented, proper $(2n-i)$-submanifold. Suppose that $N_1$ and $N_2$ intersect transversely, and let $k\in\Z$ be the number of intersection points, counted with sign.\footnote{The sign of an intersection point is positive if the concatenation of the orientations of $N_1$ and $N_2$ coincides with the orientation of $F_n(\msurf)$, and is negative otherwise.} Then $\left< [N_1],[N_2]\right>$ is equal to $k$, where $[N_1]\in H_i(F_n(\msurf))$ is the fundamental homology class of $N_1$, and $[N_2]\in H^i(F_n(\msurf))$ is the fundamental Borel-Moore homology class of $N_2$, regarded as a cohomology class by virtue of Poincare duality.
    \item Let $\scU\subset F_n(\msurf)$ be an open subspace, and let $N_1$ and $N_2$ be as above, with $N_1\subset \scU$. Then $N_2\cap \scU$ is a proper submanifold of $U$, giving a class $[N_2\cap \scU]\in H^i(U)$, and this class is the restriction of $[N_2]\in H^i(F_n(\msurf))$ along the inclusion of $U$ in $F_n(\msurf)$. In particular there is an equality of Kronecker pairings $\left<[N_1],[N_2]\right>_{F_n(\msurf)}=\left<[N_1],[N_2\cap \scU]\right>_{\scU}$.
    \item Let $N_1,N_2$ and $\scU$ be as in the previous point, and let $\scV\subset\scU$ be a closed subspace, such that $N_2\cap\overline{\scV}=\emptyset$. Let $\mscV\subset\scV$ be the interior of $\scV$. Then $N_2\cap\scU$ represents a Borel-Moore homology class in $H_{2n-i}^{BM}(\scU\setminus\mscV)$, corresponding to a relative cohomology class $[N_2\cap \scU]_{\mathrm{rel.}\mscV}\in H^i(\scU,\mscV)$. The image of $[N_2\cap \scU]_{\mathrm{rel.}\mscV}$ along the natural map
    $H^i(\scU,\mscV)\to H^i(\scU)$ is $[N_2\cap \scU]$. In particular there is an equality of Kronecker pairings
    $\left<[N_1],[N_2\cap \scU]\right>_{\scU}=\left<[N_1]_{\mathrm{rel.}\mscV},[N_2\cap \scU]_{\mathrm{rel.}\mscV}\right>_{\scU{\mathrm{rel.}\mscV}}$, where 
    $[N_1]_{\mathrm{rel.}\mscV}\in H_i(\scU,\mscV)$ is the homology class represented by $N_1$. If moreover the inclusion of pairs $(\scU,\mscV)\subset(\scU,\scV)$ is a homology isomorphism (for instance, because the inclusion of $\mscV$ in $\scV$ is a homotopy equivalence), then we can replace all occurrences of $\mscV$ by $\scV$ in the last formula.
\end{itemize}

\section{Before Fog}
\label{sec:beforefog}
We note that Theorem \ref{thm:main} in the case $n=1$ reduces to the well-known fact that $\Gamma_{g,1}$ acts non-trivially on $H_1(\surf)$; similarly, the case $n=2$ is proved in \cite{Bianchi2020}. From now on we fix $n\ge3$.

We define in this section the classes $x\in H_n(F_n(\msurf))$ and $y\in H^m(F_n(\msurf))$, and the element $\phi\in J_{g,1}(n-1)$; moreover we check $\left<x,y\right>=0$. The ``fog'' refers to the procedure used in Sections \ref{sec:puttingfog} and \ref{sec:afterfog} to prove that $\left<x,\phi^*(y)\right>=\pm1$: we will restrict our attention, in two steps, to certain subspaces $\scU'\subset\scU\subset F_n(\msurf)$, by imposing suitable conditions on the positions of the $n$ points $z_1,\dots,z_n$ of a configuration; if one of the points moves and abandons the region where it is constrained, it fades into the fog and the entire configuration tends to infinity (or, Poincare-dually, the configuration tends to the degenerate configuration)\footnote{Originally we thought of giving this paper the title ``Foggy roller-coasters'', since the pictures of curves climbing the genera of a surface and winding wildly around let us think of a roller-coaster.}.

\subsection{A list of subspaces of \texorpdfstring{$\msurf$}{M}}
\label{subsec:list}
\begin{figure}
  \centering
  \begin{tikzpicture}[xscale=.8,yscale=.8, decoration={markings,mark=at position 0.1 with {\arrow{>}}}]
  

\fill[looseness=1, yellow!70!white] (.3,0) to[out=90, in=90] ++(10.9,0) to[out=-90,in=-90] ++(-10.9,0)    to[out=0,in=180] ++(.3,0)  to[out=-90, in=-90] ++(10.2,0) to[out=90,in=90] ++(-10.2,0)    to[out=180,in=0] ++(-.3,0);
\begin{scope}[shift={(15,0)},xscale=-1]
    \fill[looseness=1, yellow!70!white] (.3,0) to[out=90, in=90] ++(10.9,0) to[out=-90,in=-90] ++(-10.9,0)    to[out=0,in=180] ++(.3,0)  to[out=-90, in=-90] ++(10.2,0) to[out=90,in=90] ++(-10.2,0)    to[out=180,in=0] ++(-.3,0);
\end{scope}

    \fill[looseness=1.1, red!80!white, opacity=.5] (11.3,-.2) to[out=90, in=90] ++(3.05,0) to[out=-90,in=-90] ++(-3.05,0) to[out=0, in=180] ++(.4,0) to[out=-90,in=-90] ++(2.3,0) to[out=90,in=90] ++(-2.3,0) to[out=180,in=0] ++(-.4,0); 

    \draw[looseness=1.1, postaction={decorate}] (11.5,-.2) to[out=90, in=90] ++(2.7,0)  node[left]{\tiny $b$};

    \draw[looseness=1, thick,blue] (.5,0) to[out=90, in=90] node[below]{\tiny$\beta$} ++(10.5,0)  to[out=-90,in=-90] ++(-10.5,0);
    \draw[looseness=.7, thick, blue] (10.4,.2) to[out=50, in=60] (9,1.1) node[left]{\tiny$\beta'$}; 
    
\begin{scope}[shift={(15,0)},xscale=-1]
    \draw[looseness=1, thick, blue] (.5,0) to[out=90, in=90] node[below]{\tiny$\alpha$} ++(10.5,0) to[out=-90,in=-90]  ++(-10.5,0);
    \draw[looseness=.7, thick, blue] (10.4,.2) to[out=50, in=60] (9,1.1) node[right]{\tiny$\alpha'$}; 
\end{scope}

    \fill[looseness=1, opacity=.8, white] (8,0) to[out=40,in=180] ++(3,2.5) to[out=0,in=140] ++(3,-2.5) to[out=180, in=0] (11.8,0) to[out=130,in=50] (10.2,0) to[out=180, in=0] (8,0);
\begin{scope}[shift={(15,0)},xscale=-1]
    \fill[looseness=1, opacity=.8, white] (8,0) to[out=40,in=180] ++(3,2.5) to[out=0,in=140] ++(3,-2.5) to[out=180, in=0] (11.8,0) to[out=130,in=50] (10.2,0) to[out=180, in=0] (8,0);
\end{scope}

    \draw[looseness=.7, thick, blue] (9,1.1) to[out=-120,in=-130] (10.4,.2);
\begin{scope}[shift={(15,0)},xscale=-1]
    \draw[looseness=.7, thick, blue] (9,1.1) to[out=-120,in=-130] (10.4,.2);
\end{scope}

    \fill[red!70!white, opacity=.6, looseness=.9] (1.5,0) to[out=90, in=90] ++(5.1,0) to[out=-90,in=-90] ++(-5.1,0) node[left]{\tiny$A$} to[out=0,in=180] ++(1.6,0) to[out=-90,in=-90] ++(2.4,0) to[out=90,in=90] ++(-2.4,0) to [out=180,in=0] ++(-1.6,0);
    \fill[looseness=1.1, red!80!white, opacity=.5] (9.3,0) to[out=90, in=90] ++(3.1,0) to[out=-90,in=-90] ++(-3.1,0) to[out=0,in=180] ++(.4,0) to[out=-90,in=-90] ++(2.3,0) to[out=90,in=90] ++(-2.3,0) to[out=180,in=0] ++(-.4,0);
    \fill[looseness=2, red!40!white] (11.1,-.9) to[out=90,in=90] ++(1,0) to[out=-90,in=-90] ++(-1,0);

  
    \draw[looseness=1, thick] (1,0) to[out=40,in=180] ++(3,2.5) to[out=0,in=140] ++(3,-2.5);
    \draw[looseness=1, thick] (3.2,0) to[out=50,in=130] ++(1.6,0);
    
    \draw[looseness=1, thick] (8,0) to[out=40,in=180] ++(3,2.5) to[out=0,in=140] ++(3,-2.5);
    \draw[looseness=1, thick] (10.2,0) to[out=50,in=130] ++(1.6,0);

    \draw[looseness=.9, postaction={decorate}] (1.75,0) node[below]{\tiny$a_1$}  to[out=90, in=90]  ++(4.6,0)  to[out=-90,in=-90] ++(-4.6,0);
    \draw[looseness=.85, postaction={decorate}] (2.1,0) node[below]{\tiny$a_2$} node[right]{...} to[out=90, in=90]  ++(4.1,0) to[out=-90,in=-90] ++(-4.1,0);
    \draw[looseness=.9, postaction={decorate}] (2.8,0) node[below]{\tiny$a_{n\!-\!2}$}  to[out=90, in=90]  ++(3.0,0) to[out=-90,in=-90] ++(-3.0,0);

    \draw[looseness=1.1, postaction={decorate}] (9.5,0)  node[left]{\tiny $c$}  to[out=-90, in=-90] ++(2.7,0) to[out=90,in=90] ++(-2.7,0);

    \fill[looseness=1, orange, opacity=.4] (7.2,-3.5) to[out=90, in=-90] (8.5,0) to [out=90, in=90] (13.2,0) to[out=-90, in=90] (13.2,-2.38) to [out=200, in=20] (12.8,-2.55) to[out=90,in=-90] (12.8,0) to[out=90, in=90] (8.9,0) to[out=-90, in=90] (7.6,-3.51) to [out=-190, in=10] (7.2,-3.5);
    \draw[looseness=1.1,] (11.5,-.2) to[out=-90, in=-90] ++(2.7,0);

    \draw[looseness=1, magenta, postaction={decorate}] (13,-2.5) to[out=90, in=-90] (13,0) to[out=90, in=90] (8.7,0) to[out=-90, in=90] node[left]{\tiny$d$} (7.43,-3.5); 

    \node at (11.7,-.7) {\tiny$\bullet$}; \node at (11.6,-.9) {\tiny$P$};
    \fill[opacity=.4, looseness=2] (11.2,-.9) to[out=90,in=90] ++(.8,0) to[out=-90,in=-90] ++(-.8,0);
    
    \draw[looseness=.8,thick] (0,0) to[out=-90,in=-90] ++(15,0) to[out=90,in=90] ++(-15,0);
    
\node at (4,-1.5) {\tiny$R_{A,\alpha}$};
\node at (10.3,-1) {\tiny$R_{c,\beta}$};
\node at (12.6,-2) {\tiny$R_{d,\alpha}$};
\node at (8,-3.2) {\tiny$R_{d,\alpha\beta}$};
\node at (13.3,-0.8) {\tiny$R_{d,b}$};

  \end{tikzpicture}
  \caption{The surface $\surf$ and some relevant subspaces of it. The regions $U_\alpha,U_\beta,U_b,U_c$ and $U_d$ (shaded but not labelled) are tubular neighbourhoods of their corresponding (labelled) curves. Each region labelled $R$ is a rectangle of intersection. The yellow region is $\cF$ and the union of the yellow and red regions is $\cE$.}
    \label{fig:arcsandcurves}
\end{figure}
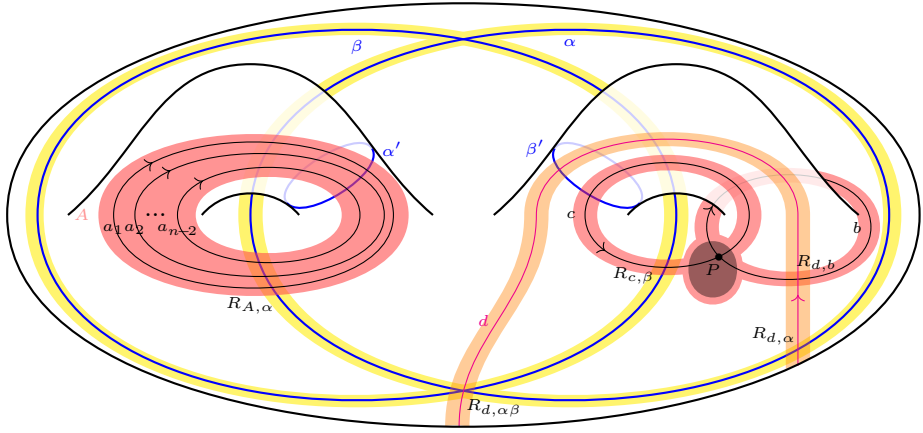
We first introduce the following subspaces of $\msurf$; the reader is invited to refer to this subsection throughout the article. All figures in this and in the next two sections depict surfaces of genus 2, but the arguments are valid for every genus $g\ge 2$.  See Figure \ref{fig:arcsandcurves}. We fix:
\begin{itemize}
    \item an open annulus $A\subset\msurf$, supporting a non-trivial first homology class, and $n-2$ oriented simple closed curves $a_1,\dots,a_{n-2}\subset A$, representing the same generator of $H_1(A)$ and occurring in this order on $A$, from left to right according to the orientation of the curves and of $A\subset\msurf$;
    \item a pair of oriented simple closed curves $b,c$ that are disjoint from $A$ and intersect once transversely at a point $P$;
    \item a simple, properly embedded oriented arc $d$, whose closure $\bar d\subset\surf$ is an embedded arc, transverse to $\del\surf$; $d$ intersects $b$ once, transversely in a point, and is disjoint from the other subspaces mentioned so far;
    \item a pair of disjoint simple closed curves $\alpha,\alpha'$, bounding a subsurface of $\msurf$ of genus 1 that contains $b$ and $c$; $\alpha$ intersects $d$ twice, transversely; $\alpha'$ intersects $A$ in an open segment, and intersects each $a_i$  once, transversely; all other intersections of $\alpha$ and $\alpha'$ with the subspaces mentioned so far are empty;
    \item a pair of disjoint simple closed curves $\beta,\beta'$, bounding a subsurface of $\msurf$ of genus 1 that contains $U$ and $\alpha'$; $\beta'$ is parallel to $b$, it intersects $c$ and $d$ transversely in a point, and is disjoint from all other subspaces mentioned so far; $\beta$ intersects $c$, $d$ and $\alpha$ transversely, respectively in one point, one point and two points, and $\beta$ is disjoint from all other subspaces mentioned so far.
\end{itemize}
We assume that $d$, $\alpha$ and $\beta$ share one intersection point, and that otherwise no three of the curves and arcs mentioned have a common intersection point.
We fix tubular neighbourhoods $U_\alpha,U_\beta,U_b,U_c,U_d$ of $\alpha,\beta,b,c,d$, respectively, in $\msurf$. We assume that each of $U_\alpha\cap A$, $U_\alpha\cap U_b$, $U_\beta\cap U_c$ and $U_\beta\cap U_d$ consists of a single rectangle, that each of $U_\alpha\cap U_\beta$ and $U_d\cap (U_\alpha\cup U_\beta)$ consists of two rectangles,
and that each other intersection between two open sets chosen among $A$ and these open neighbourhoods of curves and arcs is empty.

We denote by $\cE\subset\msurf$ the open subsurface $A\cup U_\alpha\cup U_\beta\cup U_b\cup U_c$, which is the union of the green and pink areas in Figure \ref{fig:arcsandcurves}. We denote by $\cF\subset\cE$ the open subsurface $U_\alpha\cup U_\beta$, which is the green area in Figure \ref{fig:arcsandcurves}, and by $\bar\cF$ its closure: the latter is a surface of genus $0$ with 4 boundary curves.
The \emph{outer} boundary curve of $\cF$, i.e. the one that is parallel to $\del\surf$, is denoted $\delout\cF$.

We further give names to some rectangles occurring as intersections of $\bar\cF$ with other open tubular neighbourhoods:
\begin{itemize}
    \item $R_{A,\alpha}=A\cap \bar U_\alpha$;
    \item $R_{c,\beta}=U_c\cap \bar U_\beta$;
    \item the intersection $U_d\cap\bar \cF$ consists of two rectangles $R_{d,\alpha}$ and $R_{d,\alpha\beta}$, satisfying $R_{d,\alpha\beta}\cap\beta\neq\emptyset=R_{d,\alpha}\cap \beta$.
\end{itemize}
We also denote $R_{d,b}=U_b\cap \bar U_d$. All rectangles introduced above are homeomorphic to $[0,1]\times (0,1)$, and have thus boundary homeomorphic to $\set{0,1}\times (0,1)$. We denote by $\mR_{A,\alpha}$ the interior of $R_{A,\alpha}$ and by $\del\mR_{A,\alpha}$ its boundary; similarly for the other rectangles.
Finally, we fix a small open disc $D_P\subset U_b\cap U_c$, which is the pink-grey area in Figure \ref{fig:arcsandcurves}.

We use the following convention regarding signs of intersection points of oriented curves, with particular reference to Figure \ref{fig:arcsandcurves}: let $\gamma_1,\gamma_2\subset\msurf$ be two oriented curves or arcs that are transverse to each other and represent, respectively, a first homology and a first cohomology class of a subsurface of $\msurf$; then the Kronecker pairing $\left<\gamma_1,\gamma_2\right>$ is equal to the number of points in $\gamma_1\cap\gamma_2$, counted with sign, where $Q\in\msurf$ contributes $+1$ if $\gamma_1$ meets $\gamma_2$ from right (according to the orientation of $\msurf$), and contributes $-1$ otherwise. For instance, $b$ represents a class in $H_1(\msurf)$, and $c$ and $d$ classes in $H^1(\msurf)$. Then $\left<b,c\right>=+1$, as $b$ meets $d$ in a single point, from right; similarly $\left<b,d\right>=+1$.

\subsection{Definition of \texorpdfstring{$x$}{x}}
\label{subsec:x}
\begin{figure}
  \centering
  \begin{tikzpicture}[xscale=.8,yscale=.8, decoration={markings,mark=at position 0.1 with {\arrow{>}}}]
    \draw[looseness=1.1, postaction={decorate}] (11.5,-.2) to[out=90, in=90] ++(2.7,0) ;

    \fill[looseness=1, opacity=.6, white] (8,0) to[out=40,in=180] ++(3,2.5) to[out=0,in=140] ++(3,-2.5) to[out=180, in=0] (11.8,0) to[out=130,in=50] (10.2,0) to[out=180, in=0] (8,0);
\begin{scope}[shift={(15,0)},xscale=-1]
    \fill[looseness=1, opacity=.6, white] (8,0) to[out=40,in=180] ++(3,2.5) to[out=0,in=140] ++(3,-2.5) to[out=180, in=0] (11.8,0) to[out=130,in=50] (10.2,0) to[out=180, in=0] (8,0);
\end{scope}


  
    \draw[looseness=1, thick] (1,0) to[out=40,in=180] ++(3,2.5) to[out=0,in=140] ++(3,-2.5);
    \draw[looseness=1, thick] (3.2,0) to[out=50,in=130] ++(1.6,0);
    
    \draw[looseness=1, thick] (8,0) to[out=40,in=180] ++(3,2.5) to[out=0,in=140] ++(3,-2.5);
    \draw[looseness=1, thick] (10.2,0) to[out=50,in=130] ++(1.6,0);

    \draw[looseness=.9, postaction={decorate}] (1.75,0) node{$\bullet$} node[below]{\tiny$1$} to[out=90, in=90] ++(4.6,0) to[out=-90,in=-90] ++(-4.6,0);
    \draw[looseness=.85, postaction={decorate}] (2.1,0) node{$\bullet$} node[below]{\tiny$2$} node[right]{...} to[out=90, in=90]  ++(4.1,0) to[out=-90,in=-90] ++(-4.1,0);
    \draw[looseness=.9, postaction={decorate}] (2.8,0) node{$\bullet$} node[below]{\tiny$n\!\!-\!\!2$} to[out=90, in=90] ++(3.0,0) to[out=-90,in=-90] ++(-3.0,0);

    \draw[looseness=1.1,] (11.5,-.2) to[out=-90, in=-90] node{$\bullet$} node[below]{\tiny$n\!\!-\!\!1$} ++(2.7,0);

    \fill[opacity=.4, looseness=2] (11.2,-.9) to[out=90,in=90] ++(.8,0) to[out=-90,in=-90] ++(-.8,0);

    \draw[looseness=.8,thick] (0,0) to[out=-90,in=-90] ++(15,0) to[out=90,in=90] ++(-15,0);

    \draw[looseness=1.1, postaction={decorate}] (9.5,0) node{$\bullet$} node[below]{\tiny$n$}  to[out=-90, in=-90] ++(2.7,0) to[out=90,in=90] ++(-2.7,0);

    \draw[->] (11.8,-1.1) arc (0:360:.2); \node[anchor=center] at (11.6,-1.1){$\bullet$};\node[anchor=center] at (11.6,-1.3){$\bullet$};

\begin{scope}[shift={(0,-8)}]
\node at (-1,0) {\large $-$};
    \draw[looseness=1.1, postaction={decorate}] (11.5,-.2) to[out=90, in=90] ++(2.7,0) ;

    \fill[looseness=1, opacity=.6, white] (8,0) to[out=40,in=180] ++(3,2.5) to[out=0,in=140] ++(3,-2.5) to[out=180, in=0] (11.8,0) to[out=130,in=50] (10.2,0) to[out=180, in=0] (8,0);
\begin{scope}[shift={(15,0)},xscale=-1]
    \fill[looseness=1, opacity=.6, white] (8,0) to[out=40,in=180] ++(3,2.5) to[out=0,in=140] ++(3,-2.5) to[out=180, in=0] (11.8,0) to[out=130,in=50] (10.2,0) to[out=180, in=0] (8,0);
\end{scope}


  
    \draw[looseness=1, thick] (1,0) to[out=40,in=180] ++(3,2.5) to[out=0,in=140] ++(3,-2.5);
    \draw[looseness=1, thick] (3.2,0) to[out=50,in=130] ++(1.6,0);
    
    \draw[looseness=1, thick] (8,0) to[out=40,in=180] ++(3,2.5) to[out=0,in=140] ++(3,-2.5);
    \draw[looseness=1, thick] (10.2,0) to[out=50,in=130] ++(1.6,0);

    \draw[looseness=.9, postaction={decorate}] (1.75,0) node{$\bullet$} node[below]{\tiny$1$} to[out=90, in=90] ++(4.6,0) to[out=-90,in=-90] ++(-4.6,0);
    \draw[looseness=.85, postaction={decorate}] (2.1,0) node{$\bullet$} node[below]{\tiny$2$} node[right]{...} to[out=90, in=90]  ++(4.1,0) to[out=-90,in=-90] ++(-4.1,0);
    \draw[looseness=.9, postaction={decorate}] (2.8,0) node{$\bullet$} node[below]{\tiny$n\!\!-\!\!2$} to[out=90, in=90] ++(3.0,0) to[out=-90,in=-90] ++(-3.0,0);

    \draw[looseness=1.1,] (11.5,-.2) to[out=-90, in=-90] node{$\bullet$} node[below]{\tiny$n$} ++(2.7,0);

    \fill[opacity=.4, looseness=2] (11.2,-.9) to[out=90,in=90] ++(.8,0) to[out=-90,in=-90] ++(-.8,0);

    \draw[looseness=.8,thick] (0,0) to[out=-90,in=-90] ++(15,0) to[out=90,in=90] ++(-15,0);

    \draw[looseness=1.1, postaction={decorate}] (9.5,0) node{$\bullet$} node[below]{\tiny$n\!\!-\!\!1$}  to[out=-90, in=-90] ++(2.7,0) to[out=90,in=90] ++(-2.7,0);

    \draw[->] (11.8,-1.1) arc (0:360:.2); \node[anchor=center] at (11.6,-1.1){$\bullet$};\node[anchor=center] at (11.6,-1.3){$\bullet$};
    
    \end{scope}

  \end{tikzpicture}
  \caption{The homology class $x\in  H_n(F_n(\msurf))$ is represented by a closed $n$-dimensional submanifold $N_1$ obtained by gluing the two drawn $n$-submanifolds along their equal boundary in the grey region. $N_1$ is the product of an $(n-2)$-torus $\mathbb{T}^{n-2}$ on the left genus and a $\Sigma_2=\mathbb{T}^{2}\#\mathbb{T}^{2}$ on the right genus.}
  \label{fig:Hclassx}
\end{figure}
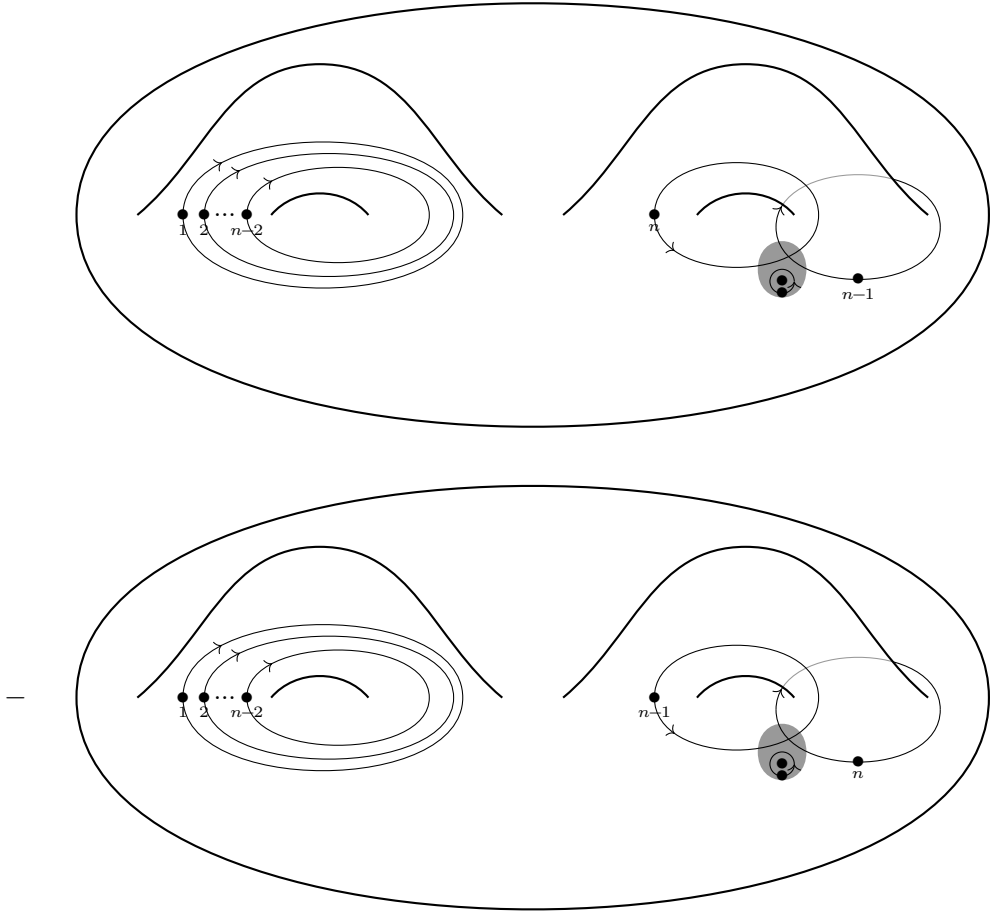

We refer to Figure \ref{fig:Hclassx}.
The class $x\in H_n(F_n(\msurf))$ is defined as the image of the fundamental homology class of a certain oriented manifold $N_1$ along a map $\iota_{N_1}\colon N_1\to F_n(\msurf)$.

We start by defining $N_1$: the first step is to define a certain family of configurations of two points, parametrised by a surface of genus 2.
\begin{definition}
\label{defn:cS2}
Let $S^1\subset\C$ denote the unit circle, and let $I=\set{z\in S^1\,|\,\Re(z)\ge 1/2}$. Let $\cS_{1,1}$ denote the space $(S^1\times S^1)\setminus (\mI\times\mI)$, and note that $\cS_{1,1}$ is a compact surface of genus $1$ with one boundary curve $\del\cS_{1,1}$; $\cS_{1,1}$ inherits an orientation from $S^1\times S^1$. Fix a homeomorphism $\rho\colon S^1\overset{\cong}{\to} \del\cS_{1,1}$, and let $\cS_2$ be the space obtained as quotient of $\cS_{1.1}\sqcup S^1\times[0,1]\sqcup\cS_{1,1}$ by identifying along $\rho$ the boundary of the first copy of $\cS_{1,1}$ with $S^1\times\set{0}$, and the boundary of the second copy of $\cS_{1,1}$ with $S^1\times\set{1}$.
\end{definition}
Note that $\cS_2$ is a topological closed, orientable surface of genus 2; we fix an orientation on $\cS_2$: note that one of the two natural embeddings $\cS_{1,1}\hookrightarrow\cS_2$ is orientation-preserving, whereas the other is orientation-reversing. Our next aim is to define a map $\iota\colon\cS_2\to\msurf^2$.
Let $I_b\subset b\cap D_P$ and $I_c\subset c\cap D_P$ be two small, closed interval neighbourhoods of $P$ in $b$ and $c$ respectively, and fix homeomorphisms of pairs $\iota_b\colon (S^1,I)\cong (b,I_b)$ and $\iota_c\colon (S^1,I)\cong(c,I_c)$. Denote by $\bar\iota_c\colon S^1\to c$ the composition of complex conjugation $S^1\overset{\cong}{\to} S^1$ and $\iota_c\colon  S^1\overset{\cong}{\to} c$.

Consider the maps
\[
\iota_b\times\iota_c,\bar\iota_c\times\iota_b\colon S^1\times S^1\to \msurf^2,
\]
and note that the preimage of the diagonal of $\msurf^2$ consists, for each of the maps, only of the point $(1,1)\in S^1\times S^1$. Thus we can define a first map
$\mathring{\iota}_{\cS_2}\colon \cS_{1,1}\sqcup\cS_{1,1}\to F_2(\msurf)$ by taking the restriction
of $\iota_b\times\iota_c$ and the restriction of $\bar\iota_c\times\iota_b$, respectively, on the two copies of $\cS_{1,1}$.

We then note that both compositions $(\iota_b\times\iota_c)\circ\rho$ and $(\bar\iota_c\times\iota_b)\circ\rho$ are maps $S^1\to\msurf^2$ with image inside $F_2(D_P)$, and that the two maps
\[
(\iota_b\times\iota_c)\circ\rho \ ,\  (\bar\iota_c\times\iota_b)\circ\rho\colon S^1\to F_2(D_P),
\]
which are in fact homotopy equivalences, are indeed homotopic. We can thus fix a homotopy $S^1\times[0,1]\to F_2(D_P)\subset F_2(\msurf)$ interpolating between the two maps, and use this homotopy to complete $\mathring{\iota}_{\cS_2}$ to a map $\iota_{\cS_2}\colon \cS_2\to F_2(\msurf)$. With a little care one can achieve that $\iota_{\cS_2}$ is an embedding, and thus consider $\cS_2$ as a subspace of $F_2(\msurf)$; we leave the details to the reader, and continue the discussion without assuming that $\iota_{\cS_2}$ is an embedding.

Note that by construction $\iota$ takes image in $F_2(b\cup c\cup D_P)$.
We define $N_1=(S^1)^{n-2}\times \cS_2$, and orient it as a product of oriented manifolds. We fix parametrisations $\iota_{a_i}\colon S^1\overset{\cong}{\to} a_i$ that are compatible with the orientations of the curve $a_i$, and define
\[
\iota_{N_1}:=\iota_{a_1}\times\dots\times\iota_{a_{n-2}}\times\iota_{\cS_2}\colon N_1\to \msurf^n.
\]
Again, we note that the image of $\iota_{N_1}$ is contained in $F_n(\msurf)$.
Loosely speaking, $\iota(N_1)$ is the subspace of configuration of $n$ ordered particles $(z_1,\dots,z_n)$ in $\msurf$, such that for $1\le i\le n-2$ the particle $z_i$ lies on $a_i$, and the particles $z_{n-1}$ and $z_{n-2}$ assemble into a configuration in the image of $\iota_{\cS_2}$. We let $x\in H_n(F_n(\msurf))$ be the image of the fundamental class of $N_1$ along $\iota_{N_1}$.

We remark that $\iota_{N_1}$ restricts to an embedding on
$(S^1)^{n-1}\times(\cS_{1,1}\sqcup\cS_{1,1})$, and that is all we are going to need later.

\subsection{Definition of \texorpdfstring{$y$}{y}}
\label{subsec:y}
\begin{figure}
  \centering
  \begin{tikzpicture}[xscale=.8,yscale=.8,decoration={markings,mark=at position 0.38 with {\arrow{>}}}]
  
  
    \draw[looseness=1, thick] (1,0) to[out=40,in=180] ++(3,2.5) to[out=0,in=140] ++(3,-2.5);
    \draw[looseness=1, thick] (3.2,0) to[out=50,in=130] ++(1.6,0);
    
    \draw[looseness=1, thick] (8,0) to[out=40,in=180] ++(3,2.5) to[out=0,in=140] ++(3,-2.5);
    \draw[looseness=1, thick] (10.2,0) to[out=50,in=130] ++(1.6,0);

    \draw[looseness=1, magenta, postaction={decorate}] (13,-2.5) to[out=90, in=-90] (13,0) to[out=90, in=90] (8.7,0) to[out=-90, in=90] node[left]{\tiny$d$} (7.4,-3.5); 
    
    \draw[looseness=.8,thick] (0,0) to[out=-90,in=-90] ++(15,0) to[out=90,in=90] ++(-15,0);

    \node at (13,-2.2){$\bullet$}; \node at (12.7,-2.3){\tiny$1$};
    \node at (13,-1.8){$\bullet$}; \node at (12.7,-1.9){\tiny$2$};
    \node at (13,-1.4){$\bullet$}; \node at (12.7,-1.5){\tiny$3$};
    \node at (12.7,-0.6){$\vdots$};
    \node at (13,0)  {$\bullet$} ; \node at (12.7,-.1){\tiny$n$};

  \end{tikzpicture}
  \caption{The cohomology class $y\in H^n(F_n(\msurf))$ is represented by the proper embedded $n$-dimensional submanifold of $F_n(\msurf)$ where the particles $1,...,n$ move along $d$ in increasing order.}
  \label{fig:Hclassy}
\end{figure}
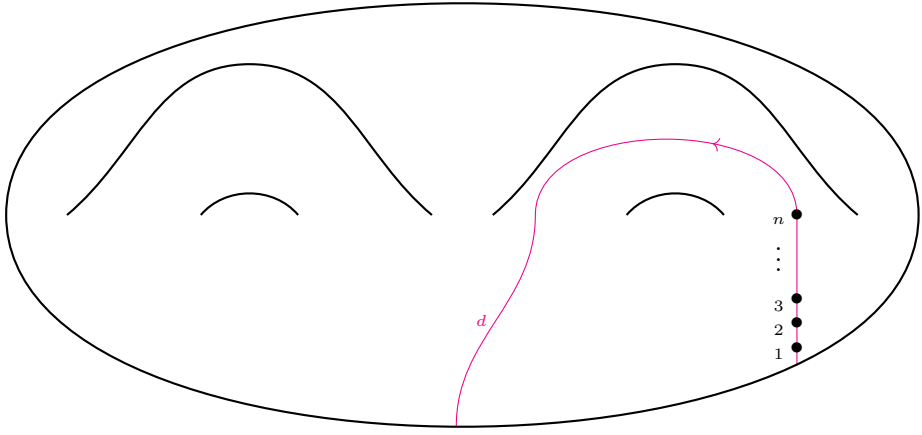
We refer to Figure \ref{fig:Hclassy}. Fix a parametrisation $\iota_d\colon[0,1]\cong \bar d$ of the closure in $\surf$ of the arc $d$, compatible with the orientation on $d$. We denote by $\Delta^n\subset[0,1]^n$ the standard simplex, consisting of all $(t_1,\dots,t_n)\in[0,1]$ with $t_1\le\dots\le t_n$. The restriction of $\iota_d^n$ gives an embedding $\iota_{\Delta^n}\colon \Delta^n\to\msurf^n$, which is a proper map since $\Delta^n$ is compact. Moreover the preimage of $F_n(\msurf)$ coincides with 
the interior $\mathring{\Delta}^n$ of $\Delta^n$. We define $N_2=\iota_{\Delta^n}(\mathring{\Delta}^n)$: it is an oriented, proper submanifold of $F_n(\msurf)$ of dimension $n$.

We define $y\in H^n(F_n(\msurf))$ as the cohomology class represented by $N_2\subset F_n(\msurf)$: it contains all configurations of $n$ distinct particles $(z_1,\dots,z_n)$ such that all $z_i$ lie on $d$, and occur on $d$ in the same order as their indices prescribe.

We notice that the image of $\iota_{N_1}$ is disjoint from $N_2$ inside $F_n(\msurf)$; this implies the following, which we state as a lemma.
\begin{lemma}
\label{lem:xdotyvanishes}
The Kronecker pairing $\left<x,y\right>$ vanishes.
\end{lemma}

\subsection{Definition of \texorpdfstring{$\phi$}{phi}}
\label{subsec:Phi}
Consider the bounding pair of Dehn twists $\phi_{\alpha}:=D_{\alpha}D_{\alpha'}^{-1}\in J_{g,1}(1)$ and the
bounding pair of Dehn twists $\phi_{\beta}:=D_{\beta}D_{\beta'}^{-1}\in J_{g,1}(1)$. See Figure \ref{fig:boundingpairs}.
    
\begin{figure}
  \centering
  \begin{tikzpicture}[xscale=.8,yscale=.8]

\fill[looseness=1, yellow!70!white] (.3,0) to[out=90, in=90] ++(10.9,0) to[out=-90,in=-90] ++(-10.9,0)    to[out=0,in=180] ++(.3,0)  to[out=-90, in=-90] ++(10.2,0) to[out=90,in=90] ++(-10.2,0)    to[out=180,in=0] ++(-.3,0);
\begin{scope}[shift={(15,0)},xscale=-1]
    \fill[looseness=1, yellow!70!white] (.3,0) to[out=90, in=90] ++(10.9,0) to[out=-90,in=-90] ++(-10.9,0)    to[out=0,in=180] ++(.3,0)  to[out=-90, in=-90] ++(10.2,0) to[out=90,in=90] ++(-10.2,0)    to[out=180,in=0] ++(-.3,0);
\end{scope}

    \draw[looseness=1, thick,blue] (.5,0) to[out=90, in=90] node[below]{\tiny$\beta$} ++(10.5,0)  to[out=-90,in=-90] ++(-10.5,0);
    \draw[looseness=.7, thick, blue] (10.4,.2) to[out=50, in=60] (9,1.1) node[left]{\tiny$\beta'$}; 
    
\begin{scope}[shift={(15,0)},xscale=-1]
    \draw[looseness=1, thick, blue] (.5,0) to[out=90, in=90] node[below]{\tiny$\alpha$} ++(10.5,0) to[out=-90,in=-90]  ++(-10.5,0);
    \draw[looseness=.7, thick, blue] (10.4,.2) to[out=50, in=60] (9,1.1) node[right]{\tiny$\alpha'$}; 
\end{scope}

    \fill[looseness=1, opacity=.6, white] (8,0) to[out=40,in=180] ++(3,2.5) to[out=0,in=140] ++(3,-2.5) to[out=180, in=0] (11.8,0) to[out=130,in=50] (10.2,0) to[out=180, in=0] (8,0);
\begin{scope}[shift={(15,0)},xscale=-1]
    \fill[looseness=1, opacity=.6, white] (8,0) to[out=40,in=180] ++(3,2.5) to[out=0,in=140] ++(3,-2.5) to[out=180, in=0] (11.8,0) to[out=130,in=50] (10.2,0) to[out=180, in=0] (8,0);
\end{scope}

  
    \draw[looseness=1, thick] (1,0) to[out=40,in=180] ++(3,2.5) to[out=0,in=140] ++(3,-2.5);
    \draw[looseness=1, thick] (3.2,0) to[out=50,in=130] ++(1.6,0);
    
    \draw[looseness=1, thick] (8,0) to[out=40,in=180] ++(3,2.5) to[out=0,in=140] ++(3,-2.5);
    \draw[looseness=1, thick] (10.2,0) to[out=50,in=130] ++(1.6,0);
    
    \draw[looseness=.7, thick, blue] (9,1.1) to[out=-120,in=-130] (10.4,.2);
    
\begin{scope}[shift={(15,0)},xscale=-1]
    \draw[looseness=.7, thick, blue] (9,1.1) to[out=-120,in=-130] (10.4,.2);
\end{scope}

    \draw[looseness=.8,thick] (0,0) to[out=-90,in=-90] ++(15,0) to[out=90,in=90] ++(-15,0);
    ;

  \end{tikzpicture}
  \caption{The mapping class $\phi$, obtained as iterated commutator of the bounding pairs $\phi_\alpha=D_{\alpha}D_{\alpha'}^{-1}$ and $\phi_\beta=D_{\beta}D_{\beta'}^{-1}$, can be represented by a diffeomorphism $\Phi$ supported on the subsurface $\cF$ (the yellow region).}
  \label{fig:boundingpairs}
  
\end{figure}
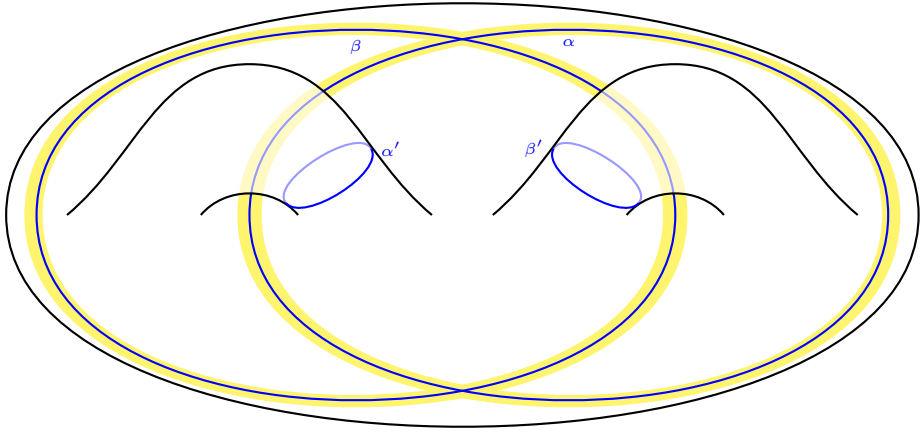

Take the commutator 
\[
\phi=[\phi_{\alpha},[\phi_{\alpha},[\dots[\phi_{\alpha},\phi_{\beta}]\dots]^{-1}
\]
where $\phi_{\alpha}$ appears $n-2$ times. Observe that $\phi\in J_{g,1}(n-1)$, by centrality of the Johnson filtration, since it is an iterated commutator of $n-1$ elements in $J_{g,1}(1)$.
Note also that both $D_{\alpha'}$ and $D_{\beta'}$ commute with $D_{\alpha}$, with $D_{\beta}$ and with each other; hence the same iterated commutator can be written as
\[
\phi=[D_{\alpha},[D_{\alpha},[\dots[D_{\alpha},D_{\beta}]\dots]^{-1}.
\]
\begin{notation}
\label{nota:Phi}
We fix a diffeomorphism $\Phi\in\Diffd(\surf)$ representing $\phi$ and supported on $\cF$, i.e. $\Phi$ fixes pointwise the complement of $\cF$. Up to an isotopy supported on $\cF$, we can assume the following:
\begin{itemize}
    \item $\Phi^{-1}(d)$ is transverse to all curves $a_i$ and to $\del A$; moreover $\Phi^{-1}(d)$ intersects $\bar A$ in closed segments that are properly embedded in $\bar A$, and each segment intersects once each of the curves $a_i$ and each component of $\del A$.
    \item $\Phi^{-1}(d)$ is transverse to $c$ and to $\del U_c$; moreover $\Phi^{-1}(d)$ intersects $\bar U_c$ in closed segments that are properly embedded in $\bar U_c$, and each segment intersects once $c$ and each component of $\del U_c$;
\end{itemize}
\end{notation}
To achieve the first requirement, it can be useful to first ensure transversality of $\Phi^{-1}(d)$ and $a_1$ by changing $\Phi$ up to isotopy, and then replace $A$ by a smaller tubular neighbourhood of $a_1$, thereby replacing also the curves $a_2,\dots,a_{n-2}$ by parallel curves. In a similar way one can achieve the second requirement.

Thanks to Lemma \ref{lem:xdotyvanishes}, we only need to check that the Kronecker pairing $\left<x,\phi^*(y)\right>$ does not vanish, in order to prove Theorem \ref{thm:main}.


\section{Putting Fog}
\label{sec:puttingfog}
\subsection{Putting fog away from \texorpdfstring{$\cF$}{cF}}
We define certain $\phi$-invariant subspaces $\scV\subset \scU\subset F_n(\msurf)$, with $\scU$ open and $\scV$ closed in $\scU$.
We then replace $x$ and $y$, respectively, by a homology class $x'\in H_i(\scU,\scV)$ and a cohomology class $y'\in H^i(\scU,\scV)$. The computation of the Kronecker pairing $\left<x,\phi^*(y)\right>$ will be reduced to the computation of the pairing $\left<x',\phi^*(y')\right>_{\scU\mathrm{rel.}\scV}$.
\begin{definition}
\label{defn:scUscV}
Recall the notation from Subsection \ref{subsec:list}. We define $\scU\subset F_n(\msurf)$ as the open subspace of configurations $(z_1,\dots,z_n)$ satisfying the following:
\begin{itemize}
    \item[($\scU$1)] all points $z_i$ lie in $\cE$;
    \item[($\scU$2)] the points $z_1,\dots,z_{n-2}$ lie in the union $\cF\cup A$;
    \item[($\scU$3)] at least one among $z_{n-1}$ and $z_n$ lies in the union $\cF\cup U_c$;
    \item[($\scU$4)] at least one among $z_{n-1}$ and $z_n$ lies in $U_b$.
\end{itemize}
We define $\scV\subset\scU$ as the relatively closed subspace of configurations $(z_1,\dots,z_n)$ satisfying in addition the following:
\begin{itemize}
    \item[($\scV$1)] at least one point $z_i$ lies in the subspace
\[
 \cE\setminus(\cF\cup U_d)=(A\setminus \mR_{A,\alpha})\cup(U_c\setminus \mR_{c,\beta})\cup(U_b\setminus \mR_{d,b}).
\]
\end{itemize}
\end{definition}
We observe the following facts.
\begin{itemize}
    \item Recall from Subsection \ref{subsec:Phi} that we have fixed a diffeomorphism $\Phi\colon\surf\to\surf$ fixing pointwise the complement of $\cF$; the action of $\Phi$ on $F_n(\msurf)$ preserves both subspaces $\scU$ and $\scV$.
    \item Recall from Subsection \ref{subsec:x} the map $\iota_{N_1}\colon N_1\to F_n(\msurf)$; then $\iota_{N_1}$ takes values inside $\scU$. Define $x'\in H_n(\scU)$ as the image along $\iota_{N_1}$ of the fundamental homology class of $N_1$; then $x'\mapsto x$ along map $H_n(\scU)\to H_n(F_n(\msurf))$ induced by the inclusion, and by the previous remark we also have $(\Phi|_{\scU})_*(x')\mapsto \phi_*(x)$ along the same map.
    \item Recall from Subsection \ref{subsec:y} the submanifold $N_2\subset F_n(\msurf)$; then $N_2$ is disjoint from $\scV$. We can thus define $y'\in H^n(\scU,\mscV)$ as the cohomology class that is Poincare dual to $N_2$; again, $y'\mapsto y|_{\scU}$ along the natural map $H^n(\scU,\mscV)\to H^n(\scU)$, and by the previous remark we also have $(\Phi|_{\scU})^*(y')\mapsto (\phi^*(y))|_{\scU}$ along the same map.
    \item The space $\scV$ is a topological submanifold with boundary of $\scU$, in particular the inclusion $\mscV\subset\scV$ is a homotopy equivalence and we can consider $y'$ as a cohomology class in $H^n(\scU,\scV)$.
\end{itemize}
From now on, by abuse of notation, we will denote the image of $x'$ along the map $H_n(\scU)\to H_n(\scU,\scV)$ also by $x'\in H_n(\scU,\scV)$.
By the principles listed in Subsection \ref{subsec:Poincare}, we have the equality $\left<x,\phi^*(y)\right>=\left<x',(\Phi|_{\scU})^*(y')\right>_{\scU\mathrm{rel.}\scV}$.

\begin{notation}
We denote by $d_{\alpha}$, $d_{\alpha\beta}$ and $d_b$ the closed segments $d\cap R_{d,\alpha}$, $d\cap R_{d,\alpha\beta}$ and $d\cap R_{d,b}$, respectively. We denote by $\md_\alpha$ the interior of $d_\alpha$, and similarly for the other segments.
\end{notation}

Consider now a configuration $(z_1,\dots,z_n)$ in the intersection $N_2\cap \scU$. By definition of $N_2$, both particles $z_{n-1}$ and $z_n$ must lie on $d$; condition $(\scU4)$ imposes then that at least one among $z_{n-1}$ and $z_n$ lies on $\md_b$, whereas ($\scU$3) imposes that at least one among $z_{n-1}$ and $z_n$ lies on $\md_\alpha\cup\md_{\alpha\beta}$.
Recall moreover that for a configuration $(z_1,\dots,z_n)$ the particles $z_{n-1}$ and $z_n$ have to appear in their natural order on $d$: we conclude that $N_2\cap\scU$ splits as the disjoint union of two closed submanifolds of $\scU$, described as follows:
\begin{itemize}
    \item $Q_n$ contains configurations $(z_1,\dots,z_n)\in N_2\cap \scU$ for which $z_1,\dots,z_{n-1}\in \md_\alpha$ and $z_n\in \md_b$;
    \item $Q_{n-1}$ contains configurations $(z_1,\dots,z_n)\in N_2\cap \scU$ for which $z_1,\dots,z_{n-2}\in \md_\alpha$, $z_{n-1}\in \md_b$ and $z_n\in \md_{\alpha,\beta}$.
\end{itemize}
The notation for $Q_{n-1}$ and $Q_n$ depends on which point lies on $\md_b$.
\begin{definition}
\label{defn:Wi}
For a natural number $i$ we denote by $W_i\subset F_i(\md_\alpha)$ the subspace of configurations $(z_1,\dots,z_i)$ such that, according to the orientation of $\md_{\alpha}$ inherited from $d$, we have $z_1<\dots <z_i$.
\end{definition}
Note that $Q_n$ is naturally homeomorphic to the product $W_{n-1}\times \md_b$. Similarly, $Q_{n-1}$ is naturally homeomorphic to the product $W_{n-2}\times \md_b\times\md_{\alpha\beta}$.

The above discussion also shows that $y'$ can be written as a sum of two cohomology classes, and each of these summands is represented by a submanifold of $\scU$ which is disjoint from $\scV$ and has, loosely speaking, the shape of a product.


\subsection{Excision of \texorpdfstring{$\mathring{\scV}$}{scV}}
Our next goal will be to use excision in order to replace the pair $(\scU,\scV)$ by another pair $(\scU',\scV')$ of subspaces of $F_n(\msurf)$, such that the latter pair also splits as a disjoint union of products of pairs of spaces, in such a way that the two cohomology classes discussed above take the form of cohomology cross products. 

\begin{definition}
Recall Definition \ref{defn:scUscV}. We let $\scU'=\scU\setminus\mathring{\scV}$, and let $\scV'=\del\scV=\scV\setminus\mathring{\scV}$.
\end{definition}
A configuration $(z_1,\dots, z_n)$ lies in $\mathring{\scV}\subset\scU$ if it satisfies conditions ($\scU$1-4) and the following condition, which is the ``open version'' of ($\scV$1):
\begin{itemize}
    \item[($\mathring{\scV}$1)] at least one point $z_i$ lies in the subspace
\[
 \cE\setminus(\bar\cF\cup \bar U_d)=(A\setminus R_{A,\alpha})\cup(U_c\setminus R_{c,\beta})\cup(U_b\setminus R_{d,b}).
\]
\end{itemize}

\begin{notation}
\label{nota:FR}
We denote $\cF^R\subset\msurf$ the union of $\cF\cup R_{A,\alpha}\cup R_{c,\beta}$, and by $\del \cF^R=\del R_{A,\alpha}\cup \del R_{c,\beta}$ its boundary.
\end{notation}
Note that $\cF^R$ is obtained from $\cF$ by adding four boundary segments.

\begin{lemma}
\label{lem:scUpcharacterisation}
The space $\scU'\subset F_n(\msurf)$ contains all configurations $(z_1,\dots,z_n)$ satisfying the following conditions:
\begin{itemize}
    \item[($\scU'$1)] all points $z_i$ lie in $\cF^R \sqcup R_{d,b}$, where we highlight the two connected components of this subspace of $\msurf$;
    \item[($\scU'$2)] the points $z_1,\dots,z_{n-2}$ lie in 
    $\cF\cup R_{A,\alpha}\subset \cF^R$
    \item[($\scU'$3)] exactly one among $z_{n-1}$ and $z_n$ lies in $\cF\cup R_{c,\beta}\subset\cF^R$;
    \item[($\scU'$4)] exactly one among $z_{n-1}$ and $z_n$ lies in $R_{d,b}$.
\end{itemize}
The subspace $\scV'\subset \scU'$ can be characterised by the following additional condition:
\begin{itemize}
\item[($\scV'$1)] at least one point $z_1,\dots,z_{n-2}$ lies on $\del R_{A,\alpha}$, or at least one point $z_{n-1},z_n$ lies on $\del R_{c,\beta}\cup\del R_{d,b}$.
\end{itemize}
\end{lemma}
\begin{proof}
Condition ($\scU'$1) is equivalent to ($\scU$1) together with the negation of ($\mathring{\scV}$1), since $\cE\setminus \big((A\setminus R_{A,\alpha})\cup(U_c\setminus R_{c,\beta})\cup(U_b\setminus R_{d,b})\big)=\cF^R\sqcup R_{d,b}$.
Similarly, condition ($\scU'$2) is equivalent to $(\scU2)\wedge\neg(\mathring{\scV}1)$, and the following equivalences hold:
\begin{itemize}
    \item $\big((\scU 3)\wedge\neg(\mathring{\scV}1)\big)\Longleftrightarrow \big(\set{z_{n-1},z_n}\cap(\cF\cup R_{c,\beta})\neq\emptyset\big)$;
    \item $\big((\scU 4)\wedge\neg(\mathring{\scV}1)\big)\Longleftrightarrow \big(\set{z_{n-1},z_n}\cap R_{d,b}\neq\emptyset\big)$.
\end{itemize}
Using that $\cF\cup R_{c,\beta}$ and $R_{d,b}$ are disjoint, we have that $\big((\scU 3)\wedge\neg(\mathring{\scV}1)\big)\wedge \big((\scU 4)\wedge\neg(\mathring{\scV}1)\big)$ is equivalent to $(\scU'3)\wedge(\scU'4)$.

Finally, condition ($\scV$1) together with the negation of ($\mathring{scV}$1) is equivalent to condition ($\scV'$1).
\end{proof}

\begin{notation}
\label{nota:FSZ}
For a finite set $S$ and a space $Z$ we denote by $Z^S$ the space of all functions $z\colon S\to Z$, and by $F_S(Z)\subset Z^S$ the subspace of injective functions. For $i\in S$ and $z\in Z^S$ we usually denote $z_i= z(i)\in Z$; we regard elements of $F_S(Z)$ as $S$-labeled configurations of distinct points in $Z$.
Clearly, when $S=\set{1,\dots,n}$ we have canonical identifications $Z^S=Z^n$ and $F_S(Z)\cong F_n(Z)$.
\end{notation}
We note that the pair $(\scU',\scV')$ decomposes as a disjoint union of two pairs $(\scU'_{n-1},\scV'_{n-1})$ and $(\scU'_n,\scV'_n)$, where for $i=n-1,n$ we let $(\scU'_i,\scV'_i)$ contain those configurations $(z_1,\dots,z_n)\in(\scU'\scV')$ for which $z_i\in R_{d,b}$.
\begin{definition}
For $i=n-1,n$, we let $\scX_i\subset F_{\set{1,\dots,n}\setminus\set{i}}(\cF^R)$ be the subspace of configurations satisfying ($\scU'$2) and $z_{2n-1-i}\in \cF\cup R_{c,\beta}$ (the latter condition is a specification of ($\scU'$3)), and we let $\scY_i\subset\scX_i$ be the subspace of configurations satisfying $z_{2n-1-i}\in \del R_{c,\beta}$ (the latter condition is a specification of ($\scV'$1)).
\end{definition}
The following is a straightforward corollary of Lemma \ref{lem:scUpcharacterisation}
\begin{corollary}
\label{cor:homeocouple}
For $i=n-1,n$,
\begin{itemize}
\item  we have a canonical homeomorphism of pairs
\[
(\scU'_i,\scV'_i) \cong (\scX_i,\scY_i)\times (F_{\set{i}}(R_{d,b}),F_{\set{i}}(\del R_{d,b}))
\]
\item the action of the diffeomorphism $\Phi$ on $F_n(\msurf)$ preserves each of the subspaces $\scU'_i$ and $\scV'_i$; the action of $\Phi$ on $\scU'_i$ is obtained by crossing the natural action of $\Phi$ on $\scX_i$ (as $\Phi$-invariant subspace of $F_{\set{1,\dots,n}\setminus\set{i}}(\msurf)$) with the identity of $R_{d,b}$.
\end{itemize}
\end{corollary}

\begin{definition}
For $i=n-1,n$ we define $y'_i\in H^n(\scU'_i,\scV'_i)$ as the cohomology class represented by the oriented submanifold $Q_i$. See Figure \ref{fig:Hclassyprime}.
\end{definition}
\begin{figure}
  \centering
  \begin{tikzpicture}[xscale=.8,yscale=.8]
  
  \node at (-1,0) {\large $y'_{n-1}=-$};
  \node[blue] at (11.5,-2) {\large $r_{n-1}$};
  \node[magenta] at (12.5,-.7) {\large $s^{\vee}$};

  \fill[looseness=1, yellow!70!white] (.3,0) to[out=90, in=90] ++(10.9,0) to[out=-90,in=-90] ++(-10.9,0)    to[out=0,in=180] ++(.5,0)  to[out=-90, in=-90] ++(9.8,0) to[out=90,in=90] ++(-9.8,0)    to[out=180,in=0] ++(-.3,0);
\begin{scope}[shift={(15,0)},xscale=-1]
    \fill[looseness=1, yellow!70!white] (.3,0) to[out=90, in=90] ++(10.9,0) to[out=-90,in=-90] ++(-10.9,0)    to[out=0,in=180] ++(.5,0)  to[out=-90, in=-90] ++(9.8,0) to[out=90,in=90] ++(-9.8,0)    to[out=180,in=0] ++(-.3,0);
\end{scope}
    \fill[looseness=1, opacity=.6, white] (8,0) to[out=40,in=180] ++(3,2.5) to[out=0,in=140] ++(3,-2.5) to[out=180, in=0] (11.8,0) to[out=130,in=50] (10.2,0) to[out=180, in=0] (8,0);
\begin{scope}[shift={(15,0)},xscale=-1]
    \fill[looseness=1, opacity=.6, white] (8,0) to[out=40,in=180] ++(3,2.5) to[out=0,in=140] ++(3,-2.5) to[out=180, in=0] (11.8,0) to[out=130,in=50] (10.2,0) to[out=180, in=0] (8,0);
\end{scope}

  
    \draw[looseness=1, thick] (1,0) to[out=40,in=180] ++(3,2.5) to[out=0,in=140] ++(3,-2.5);
    \draw[looseness=1, thick] (3.2,0) to[out=50,in=130] ++(1.6,0);
    
    \draw[looseness=1, thick] (8,0) to[out=40,in=180] ++(3,2.5) to[out=0,in=140] ++(3,-2.5);
    \draw[looseness=1, thick] (10.2,0) to[out=50,in=130] ++(1.6,0);

\begin{scope}
    \clip[looseness=1] (14.75,0) to[out=-90, in=-90] ++(-11.0,0) to[out=0,in=180] ++(.6,0)  to[out=-90, in=-90] ++(9.8,0) to[out=0,in=180] ++(.6,0);
    \draw[looseness=1, blue, postaction={decorate}] (13,-2.5) to[out=90, in=-90] (13,0) to[out=90, in=90] (8.7,0) to[out=-90, in=90] (7.4,-3.5); 
\end{scope}
\begin{scope}
    \clip[looseness=1] (7.2,-3.5) to[out=90, in=-90] (8.5,0) to [out=90, in=90] (13.25,0) to[out=-90, in=90] (13.25,-2.38) to [out=200, in=20] (12.75,-2.55) to[out=90,in=-90] (12.75,0) to[out=90, in=90] (8.9,0) to[out=-90, in=90] (7.6,-3.51) to [out=-190, in=10] (7.2,-3.5); 
     \clip[looseness=1.1] (14.3,-.2) to[out=-90,in=-90] ++(-3.1,0) to[out=0, in=180] ++(.7,0) to[out=-90,in=-90] ++(1.7,0) to[out=0,in=0] ++(.7,0); 
    \fill[looseness=1.1, red!80!white, opacity=.5] (14.3,-.2) to[out=-90,in=-90] ++(-3.1,0) to[out=0, in=180] ++(.7,0) to[out=-90,in=-90] ++(1.7,0) to[out=0,in=0] ++(.7,0); 
    \draw[looseness=1, magenta, postaction={decorate}] (13,-2.5) to[out=90, in=-90] (13,0) to[out=90, in=90] (8.7,0) to[out=-90, in=90] node[left]{\tiny$d$} (7.4,-3.5); 
\end{scope}

    \draw[looseness=.8,thick] (0,0) to[out=-90,in=-90] ++(15,0) to[out=90,in=90] ++(-15,0);

    \node[blue] at (13,-2.37){\small$\bullet$}; \node[blue] at (12.8,-2.4){\tiny$1$};
    \node[blue] at (13,-2.2){\small$\bullet$};
    \node[blue] at (13,-2.03){\small$\bullet$}; \node[blue] at (12.6,-2.0){\tiny$n\!\!-\!\!2$};
    \node[blue] at (7.5,-2.9){\small$\bullet$}; \node[blue] at (7.2,-2.9){\tiny$n$};
    \node[magenta] at (13,-1)  {\small$\bullet$} ; \node[magenta] at (12.5,-1.1){\tiny$n\!\!-\!\!1$};


\begin{scope}[shift={(0,-8)}]
   \node at (-1,0) {\large $y'_n=$};
  \node[blue] at (11.5,-2) {\large $r_n$};
  \node[magenta] at (12.5,-.7) {\large $s^{\vee}$};

  \fill[looseness=1, yellow!70!white] (.3,0) to[out=90, in=90] ++(10.9,0) to[out=-90,in=-90] ++(-10.9,0)    to[out=0,in=180] ++(.5,0)  to[out=-90, in=-90] ++(9.8,0) to[out=90,in=90] ++(-9.8,0)    to[out=180,in=0] ++(-.3,0);
\begin{scope}[shift={(15,0)},xscale=-1]
    \fill[looseness=1, yellow!70!white] (.3,0) to[out=90, in=90] ++(10.9,0) to[out=-90,in=-90] ++(-10.9,0)    to[out=0,in=180] ++(.5,0)  to[out=-90, in=-90] ++(9.8,0) to[out=90,in=90] ++(-9.8,0)    to[out=180,in=0] ++(-.3,0);
\end{scope}
    \fill[looseness=1, opacity=.6, white] (8,0) to[out=40,in=180] ++(3,2.5) to[out=0,in=140] ++(3,-2.5) to[out=180, in=0] (11.8,0) to[out=130,in=50] (10.2,0) to[out=180, in=0] (8,0);
\begin{scope}[shift={(15,0)},xscale=-1]
    \fill[looseness=1, opacity=.6, white] (8,0) to[out=40,in=180] ++(3,2.5) to[out=0,in=140] ++(3,-2.5) to[out=180, in=0] (11.8,0) to[out=130,in=50] (10.2,0) to[out=180, in=0] (8,0);
\end{scope}

  
    \draw[looseness=1, thick] (1,0) to[out=40,in=180] ++(3,2.5) to[out=0,in=140] ++(3,-2.5);
    \draw[looseness=1, thick] (3.2,0) to[out=50,in=130] ++(1.6,0);
    
    \draw[looseness=1, thick] (8,0) to[out=40,in=180] ++(3,2.5) to[out=0,in=140] ++(3,-2.5);
    \draw[looseness=1, thick] (10.2,0) to[out=50,in=130] ++(1.6,0);

\begin{scope}
    \clip[looseness=1] (14.75,0) to[out=-90, in=-90] ++(-11.0,0) to[out=0,in=180] ++(.6,0)  to[out=-90, in=-90] ++(9.8,0) to[out=0,in=180] ++(.6,0);
    \draw[looseness=1, blue, postaction={decorate}] (13,-2.5) to[out=90, in=-90] (13,0) to[out=90, in=90] (8.7,0) to[out=-90, in=90] (7.4,-3.5); 
\end{scope}
\begin{scope}
    \clip[looseness=1] (7.2,-3.5) to[out=90, in=-90] (8.5,0) to [out=90, in=90] (13.25,0) to[out=-90, in=90] (13.25,-2.38) to [out=200, in=20] (12.75,-2.55) to[out=90,in=-90] (12.75,0) to[out=90, in=90] (8.9,0) to[out=-90, in=90] (7.6,-3.51) to [out=-190, in=10] (7.2,-3.5); 
     \clip[looseness=1.1] (14.3,-.2) to[out=-90,in=-90] ++(-3.1,0) to[out=0, in=180] ++(.7,0) to[out=-90,in=-90] ++(1.7,0) to[out=0,in=0] ++(.7,0); 
    \fill[looseness=1.1, red!80!white, opacity=.5] (14.3,-.2) to[out=-90,in=-90] ++(-3.1,0) to[out=0, in=180] ++(.7,0) to[out=-90,in=-90] ++(1.7,0) to[out=0,in=0] ++(.7,0); 
    \draw[looseness=1, magenta, postaction={decorate}] (13,-2.5) to[out=90, in=-90] (13,0) to[out=90, in=90] (8.7,0) to[out=-90, in=90] node[left]{\tiny$d$} (7.4,-3.5); 
\end{scope}

    \draw[looseness=.8,thick] (0,0) to[out=-90,in=-90] ++(15,0) to[out=90,in=90] ++(-15,0);

    \node[blue] at (13,-2.37){\small$\bullet$}; \node[blue] at (12.8,-2.4){\tiny$1$};
    \node[blue] at (13,-2.2){\small$\bullet$};
    \node[blue] at (13,-2.03){\small$\bullet$}; \node[blue] at (12.6,-2.0){\tiny$n\!\!-\!\!1$};
    \node[magenta] at (13,-1)  {\small$\bullet$} ; \node[magenta] at (12.6,-1.1){\tiny$n$};
   
\end{scope}

  \end{tikzpicture}
  \caption{The cohomology class $y'$ splits as a sum $y'_{n-1}+y'_n$, and each summand further factors as a cross product so that $y'_{n-1}=-r_{n-1}\times s^{\vee}$ and $y'_n=r_n\times s^{\vee}$.}
  \label{fig:Hclassyprime}
\end{figure}
In the previous definition we use that $\scV'_i$, which is closed in $\scU'_i$, admits an open neighbourhood in $\scU'_i$ which is disjoint from $Q_i$ and is homotopy equivalent to $\scV'_i$ itself: thus we are allowed to represent a cohomology class of $(\scU'_i,\scV'_i)$ by the proper, oriented submanifold $Q_i\subset\scU'_i$, which rather supports a Borel-Moore homology class.

The disjoint union and excision isomorphisms combine as an isomorphism $H^n(\scU'_n,\scV'_n)\oplus H^n(\scU'_{n-1},\scV'_{n-1})\cong H^n(\scU',\scV')\cong H^n(\scU,\scV)$, and the images of $y'_{n}$ and $y'_{n-1}$ are two cohomology classes in $H^n(\scU,\scV)$ whose sum is $y'$. Therefore, in order to compute the Kronecker pairing $\left<x,\phi^*y\right>=\left<x',\Phi^*(y')\right>_{\scU\mathrm{rel.}\scV}$, we can first compute the homology class corresponding to $x'$ in $H_n(\scU'_n,\scV'_n)\oplus H_n(\scU'_{n-1},\scV'_{n-1})$, and then take the Kronecker pairing of this class with $\Phi^*(y'_{n})+\Phi^*(y'_{n-1})$.

\subsection{Replacing \texorpdfstring{$x'$}{x'} by an excided homology class}
\begin{definition}
Let $I=[0,1]$, and let $N_1'=[0,1]^n$. Fix orientation-preserving parametrisations
\[
\sigma_{a_1},\dots,\sigma_{a_{n-2}},\sigma_c,\sigma_b\colon I\to \msurf
\]
of each closed segment $a_1\cap R_{A,\alpha},\dots, a_{n-2}\cap R_{A,\alpha}, c\cap R_{c,\beta}$ and $b\cap R_{d,b}$, where each of the latter segments inherits an orientation from the oriented curve or arc in which it is contained.

We define a map $\iota_{N_1'}\colon \set{n-1,n}\times N_1'\to \msurf^n$ as follows:
\begin{itemize}
    \item on $\set{n-1}\times N_1'\cong N_1'$ we take the cartesian product $\sigma_{a_1}\times\dots\times\sigma_{a_{n-2}}\times\sigma_b\times\sigma_c$;
    \item on $\set{n}\times N_1'\cong N_1'$ we take the cartesian product $\sigma_{a_1}\times\dots\times\sigma_{a_{n-2}}\times\bar\sigma_c\times\sigma_b$, where $\bar\sigma_c$ is the composition of a fixed orientation-reversing homeomorphism $[0,1]\overset{\cong}{\to}[0,1]$ and the map $\sigma_c$.
\end{itemize}
\end{definition}
Note that the restriction of $\iota_{N_1'}$ on $\set{i}\times N_1'$ uses $\sigma_b$ on the $i$\textsuperscript{th} coordinate.
We observe that $\iota_{N_1'}$ has image inside $\scU'$, and sends $\del N_1'$ inside $\scV'$.
Moreover $\iota_{N_1'}$, considered as a map $\set{n-1,n}\times N_1'\to \msurf^n$, factors (uniquely) as the composition of an orientation-preserving embedding $\varepsilon\colon \set{n-1,n}\times N'_1\hookrightarrow N_1$ followed by the map $\iota_{N_1}\colon N_1\to \msurf^n$, and the difference $N_1\setminus\varepsilon(\set{n-1,n}\times\mathring{N}_1')$ is sent inside $\scV$ by $\iota_{N_1}$. By excision we obtain the following lemma.
\begin{lemma}
\label{lem:xsplitting}
The relative fundamental class of $\set{n-1,n}\times (N'_1,\del N'_1)$ is sent along $\iota_{N'_1}$ to the class $x'\in H_n(\scU,\scV)$, i.e. to the image of the fundamental class of $N_1$ along $\iota_{N_1}\colon N_1\to\scU$.
\end{lemma}

\begin{definition}
For $i=n-1,n$ we denote by $x'_i\in H_n(\scU'_i,\scV'_i)$ the image along $\iota_{N_1'}$ of the relative fundamental class of $\set{i}\times(N_1',\del N_1')$. See Figure \ref{fig:Hclassxprime}
\end{definition}
\begin{figure}
  \centering
  \begin{tikzpicture}[xscale=.8,yscale=.8, decoration={markings,mark=at position 0.1 with {\arrow{>}}}]
  \node at (-1,0) {\large $x'_{n-1}=$};
  \node[blue] at (4,-1.7) {\large $t_{n-1}$};
  \node[magenta] at (12.5,-1.7) {\large $s$};

  \fill[looseness=1, yellow!70!white] (.3,0) to[out=90, in=90] ++(10.9,0) to[out=-90,in=-90] ++(-10.9,0)    to[out=0,in=180] ++(.3,0)  to[out=-90, in=-90] ++(10.2,0) to[out=90,in=90] ++(-10.2,0)    to[out=180,in=0] ++(-.3,0);
\begin{scope}[shift={(15,0)},xscale=-1]
    \fill[looseness=1, yellow!70!white] (.3,0) to[out=90, in=90] ++(10.9,0) to[out=-90,in=-90] ++(-10.9,0)    to[out=0,in=180] ++(.3,0)  to[out=-90, in=-90] ++(10.2,0) to[out=90,in=90] ++(-10.2,0)    to[out=180,in=0] ++(-.3,0);
\end{scope}

\begin{scope}
    \clip[looseness=1] (7.2,-3.5) to[out=90, in=-90] (8.5,0) to [out=90, in=90] (13.25,0) to[out=-90, in=90] (13.25,-2.38) to [out=200, in=20] (12.75,-2.55) to[out=90,in=-90] (12.75,0) to[out=90, in=90] (8.9,0) to[out=-90, in=90] (7.6,-3.51) to [out=-190, in=10] (7.2,-3.5); 
    \fill[looseness=1.1, red!80!white, opacity=.5] (14.35,-.2) to[out=-90,in=-90] ++(-3.05,0) to[out=0, in=180] ++(.4,0) to[out=-90,in=-90] ++(2.3,0) to[out=0,in=0] ++(.4,0); 
    \draw[looseness=1.1,magenta] (11.5,-.2) to[out=-90, in=-90]  ++(2.7,0);
\end{scope}
\node[magenta] at (13,-1.1) {$\bullet$}; \node[magenta] at (13,-1.4) {\tiny$n\!\!-\!\!1$};

    \fill[looseness=1, opacity=.6, white] (8,0) to[out=40,in=180] ++(3,2.5) to[out=0,in=140] ++(3,-2.5) to[out=180, in=0] (11.8,0) to[out=130,in=50] (10.2,0) to[out=180, in=0] (8,0);
\begin{scope}[shift={(15,0)},xscale=-1]
    \fill[looseness=1, opacity=.6, white] (8,0) to[out=40,in=180] ++(3,2.5) to[out=0,in=140] ++(3,-2.5) to[out=180, in=0] (11.8,0) to[out=130,in=50] (10.2,0) to[out=180, in=0] (8,0);
\end{scope}

  
    \draw[looseness=1, thick] (1,0) to[out=40,in=180] ++(3,2.5) to[out=0,in=140] ++(3,-2.5);
    \draw[looseness=1, thick] (3.2,0) to[out=50,in=130] ++(1.6,0);
    
    \draw[looseness=1, thick] (8,0) to[out=40,in=180] ++(3,2.5) to[out=0,in=140] ++(3,-2.5);
    \draw[looseness=1, thick] (10.2,0) to[out=50,in=130] ++(1.6,0);
    
\begin{scope}
\clip[looseness=1] (14.75,0) to[out=-90, in=-90] ++(-11.0,0) to[out=0,in=180] ++(.5,0)  to[out=-90, in=-90] ++(10.1,0) to[out=0,in=180] ++(.4,0);
      \fill[red!70!white, opacity=.6, looseness=.9] (1.5,0) to[out=90, in=90] ++(5.1,0) to[out=-90,in=-90] ++(-5.1,0) to[out=0,in=180] ++(1.6,0) to[out=-90,in=-90] ++(2.4,0) to[out=90,in=90] ++(-2.4,0) to [out=180,in=0] ++(-1.6,0);
    \draw[looseness=.9, postaction={decorate},blue] (1.75,0)  to[out=90, in=90] ++(4.6,0) to[out=-90,in=-90] ++(-4.6,0);
    \draw[looseness=.85, postaction={decorate},blue] (2.1,0) to[out=90, in=90]  ++(4.1,0) to[out=-90,in=-90] ++(-4.1,0);
    \draw[looseness=.9, postaction={decorate},blue] (2.8,0) to[out=90, in=90] ++(3.0,0) to[out=-90,in=-90] ++(-3.0,0);
\end{scope}
\node[blue] at (4.4,-1.2) {$\bullet$}; \node[blue] at (4.8,-1.2) {\tiny$1$};
\node[blue] at (4.3,-1) {$\bullet$}; \node[blue] at (4.65,-1) {\tiny$2$};
\node[blue] at (4.2,-.8) {$\bullet$}; \node[blue] at (4.7,-.7) {\tiny$n\!\!-\!\!2$};

\begin{scope}
\clip[looseness=1] (.25,0) to[out=-90, in=-90] ++(11.0,0) to[out=180,in=0] ++(-.5,0) to[out=-90,in=-90] ++(-10.1,0) to[out=180,in=0] ++(-.4,0);
\fill[looseness=1.1, red!80!white, opacity=.5] (9.3,0) to[out=-90, in=-90] ++(3.1,0) to[out=-180,in=0] ++(-.4,0) to[out=-90,in=-90] ++(-2.3,0) to[out=180,in=0] ++(-.4,0); 

    \draw[looseness=1.1, postaction={decorate},blue] (9.5,0)  to[out=-90, in=-90] node{$\bullet$} ++(2.7,0) to[out=90,in=90] ++(-2.7,0);
\end{scope}
\node[blue] at (11,-1.2) {\tiny$n$} ;

    \fill[looseness=1, orange, opacity=.2] (7.2,-3.5) to[out=90, in=-90] (8.5,0) to [out=90, in=90] (13.2,0) to[out=-90, in=90] (13.2,-2.38) to [out=200, in=20] (12.8,-2.55) to[out=90,in=-90] (12.8,0) to[out=90, in=90] (8.9,0) to[out=-90, in=90] (7.6,-3.51) to [out=-190, in=10] (7.2,-3.5);
    
    \draw[looseness=.8,thick] (0,0) to[out=-90,in=-90] ++(15,0) to[out=90,in=90] ++(-15,0);


\begin{scope}[shift={(0,-8)}]
\node at (-1,0) {\large $x'_n=-$};

  \node[blue] at (4,-1.7) {\large $t_n$};
  \node[magenta] at (12.5,-1.7) {\large $s$};

  \fill[looseness=1, yellow!70!white] (.3,0) to[out=90, in=90] ++(10.9,0) to[out=-90,in=-90] ++(-10.9,0)    to[out=0,in=180] ++(.3,0)  to[out=-90, in=-90] ++(10.2,0) to[out=90,in=90] ++(-10.2,0)    to[out=180,in=0] ++(-.3,0);
\begin{scope}[shift={(15,0)},xscale=-1]
    \fill[looseness=1, yellow!70!white] (.3,0) to[out=90, in=90] ++(10.9,0) to[out=-90,in=-90] ++(-10.9,0)    to[out=0,in=180] ++(.3,0)  to[out=-90, in=-90] ++(10.2,0) to[out=90,in=90] ++(-10.2,0)    to[out=180,in=0] ++(-.3,0);
\end{scope}

\begin{scope}
    \clip[looseness=1] (7.2,-3.5) to[out=90, in=-90] (8.5,0) to [out=90, in=90] (13.25,0) to[out=-90, in=90] (13.25,-2.38) to [out=200, in=20] (12.75,-2.55) to[out=90,in=-90] (12.75,0) to[out=90, in=90] (8.9,0) to[out=-90, in=90] (7.6,-3.51) to [out=-190, in=10] (7.2,-3.5); 
    \fill[looseness=1.1, red!80!white, opacity=.5] (14.35,-.2) to[out=-90,in=-90] ++(-3.05,0) to[out=0, in=180] ++(.4,0) to[out=-90,in=-90] ++(2.3,0) to[out=0,in=0] ++(.4,0); 
    \draw[looseness=1.1,magenta] (11.5,-.2) to[out=-90, in=-90]  ++(2.7,0);
\end{scope}
\node[magenta] at (13,-1.1) {$\bullet$}; \node[magenta] at (13,-1.4) {\tiny$n$};

    \fill[looseness=1, opacity=.6, white] (8,0) to[out=40,in=180] ++(3,2.5) to[out=0,in=140] ++(3,-2.5) to[out=180, in=0] (11.8,0) to[out=130,in=50] (10.2,0) to[out=180, in=0] (8,0);
\begin{scope}[shift={(15,0)},xscale=-1]
    \fill[looseness=1, opacity=.6, white] (8,0) to[out=40,in=180] ++(3,2.5) to[out=0,in=140] ++(3,-2.5) to[out=180, in=0] (11.8,0) to[out=130,in=50] (10.2,0) to[out=180, in=0] (8,0);
\end{scope}


  
    \draw[looseness=1, thick] (1,0) to[out=40,in=180] ++(3,2.5) to[out=0,in=140] ++(3,-2.5);
    \draw[looseness=1, thick] (3.2,0) to[out=50,in=130] ++(1.6,0);
    
    \draw[looseness=1, thick] (8,0) to[out=40,in=180] ++(3,2.5) to[out=0,in=140] ++(3,-2.5);
    \draw[looseness=1, thick] (10.2,0) to[out=50,in=130] ++(1.6,0);
    
\begin{scope}
\clip[looseness=1] (14.75,0) to[out=-90, in=-90] ++(-11.0,0) to[out=0,in=180] ++(.5,0)  to[out=-90, in=-90] ++(10.1,0) to[out=0,in=180] ++(.4,0);
      \fill[red!70!white, opacity=.6, looseness=.9] (1.5,0) to[out=90, in=90] ++(5.1,0) to[out=-90,in=-90] ++(-5.1,0) to[out=0,in=180] ++(1.6,0) to[out=-90,in=-90] ++(2.4,0) to[out=90,in=90] ++(-2.4,0) to [out=180,in=0] ++(-1.6,0);
    \draw[looseness=.9, postaction={decorate},blue] (1.75,0)  to[out=90, in=90] ++(4.6,0) to[out=-90,in=-90] ++(-4.6,0);
    \draw[looseness=.85, postaction={decorate},blue] (2.1,0) to[out=90, in=90]  ++(4.1,0) to[out=-90,in=-90] ++(-4.1,0);
    \draw[looseness=.9, postaction={decorate},blue] (2.8,0) to[out=90, in=90] ++(3.0,0) to[out=-90,in=-90] ++(-3.0,0);
\end{scope}
\node[blue] at (4.4,-1.2) {$\bullet$}; \node[blue] at (4.8,-1.2) {\tiny$1$};
\node[blue] at (4.3,-1) {$\bullet$}; \node[blue] at (4.65,-1) {\tiny$2$};
\node[blue] at (4.2,-.8) {$\bullet$}; \node[blue] at (4.7,-.7) {\tiny$n\!\!-\!\!2$};

\begin{scope}
\clip[looseness=1] (.25,0) to[out=-90, in=-90] ++(11.0,0) to[out=180,in=0] ++(-.5,0) to[out=-90,in=-90] ++(-10.1,0) to[out=180,in=0] ++(-.4,0);
\fill[looseness=1.1, red!80!white, opacity=.5] (9.3,0) to[out=-90, in=-90] ++(3.1,0) to[out=-180,in=0] ++(-.4,0) to[out=-90,in=-90] ++(-2.3,0) to[out=180,in=0] ++(-.4,0); 

    \draw[looseness=1.1, postaction={decorate},blue] (9.5,0)  to[out=-90, in=-90] node{$\bullet$} ++(2.7,0) to[out=90,in=90] ++(-2.7,0);
\end{scope}
\node[blue] at (11,-1.2) {\tiny$n\!\!-\!\!1$} ;

    \fill[looseness=1, orange, opacity=.2] (7.2,-3.5) to[out=90, in=-90] (8.5,0) to [out=90, in=90] (13.2,0) to[out=-90, in=90] (13.2,-2.38) to [out=200, in=20] (12.8,-2.55) to[out=90,in=-90] (12.8,0) to[out=90, in=90] (8.9,0) to[out=-90, in=90] (7.6,-3.51) to [out=-190, in=10] (7.2,-3.5);
    
    \draw[looseness=.8,thick] (0,0) to[out=-90,in=-90] ++(15,0) to[out=90,in=90] ++(-15,0);

\end{scope}

  \end{tikzpicture}
  \caption{The homology class $x'$ splits as a sum $x'_{n-1}+x'_n$, and each summand further factors as a cross product so that $x'_{n-1}=-t_{n-1}\times s$ and $x'_n=t_n\times s$.}
  \label{fig:Hclassxprime}
\end{figure}
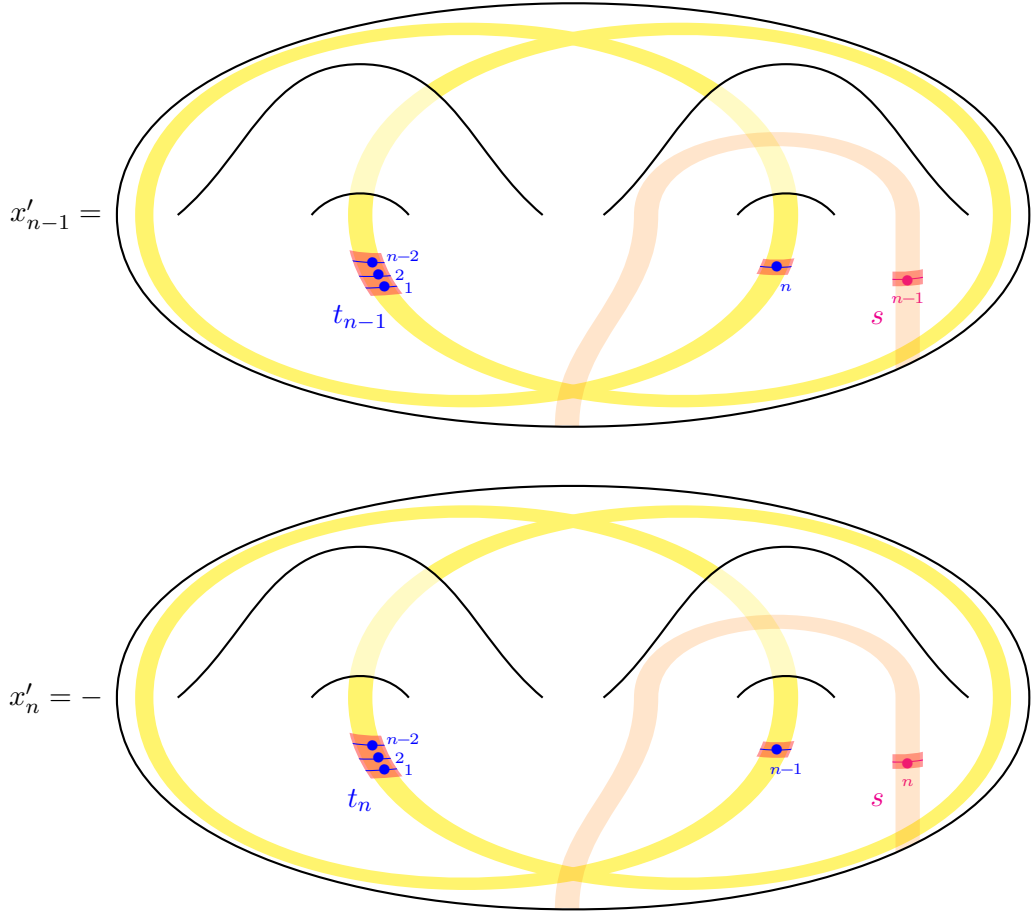
By Lemma \ref{lem:xsplitting} the sum $x'_{n-1}+ x'_n$ corresponds to 
$x'$ along the composite isomorphism
$H_n(\scU'_{n-1},\scV'_{n-1})\oplus H_n(\scU'_n,\scV'_n) \cong  H_n(\scU',\scV')\cong H_n(\scU,\scV)$.
Together with the discussion of the previous subsection, we obtain the following lemma.
\begin{lemma}
\label{lem:KrxpiPhiypi}
The Kronecker pairing $\left<x,\phi^*y\right>$ can be computed as
\[
\left<x'_{n-1},\Phi^*(y'_{n-1})\right>_{\scU'_{n-1}\mathrm{rel.}\scV'_{n-1}}+\left<x'_{n},\Phi^*(y'_{n})\right>_{\scU'_{n}\mathrm{rel.}\scV'_{n}}.
\]
\end{lemma}

\subsection{Decomposing the classes \texorpdfstring{$x_i'$}{xi'} as products}
We can make a step further. We identify the cube $\set{n-1}\times N_1'$ with the product $I^{\set{1,\dots,n-2,n}}\times I^{\set{n-1}}$
and the cube $\set{n}\times N_1'$ with the product $I^{\set{1,\dots,n-2,n-1}}\times I^{\set{n}}$. Note that the first identification is orientation-reversing, as we swap the coordinates relative to the indices $n-1$ and $n$, whereas the second is orientation-preserving.

The map $\iota_{N_1'}$ restricts to product maps on both $\set{n-1}\times N_1'$
and $\set{n}\times N_1'$, with images in $\scU'_{n-1}$ and $\scU'_n$ respectively, namely the maps
\[
(\sigma_{a_1}\times\dots\times\sigma_{a_{n-2}}\times\sigma_c)\times \sigma_b\colon I^{\set{1,\dots,n-2,n}}\times I^{\set{n-1}}\to \scX_{n-1}\times F_{\set{n-1}}(R_b,d);
\]
\[
(\sigma_{a_1}\times\dots\times\sigma_{a_{n-2}}\times\bar\sigma_c)\times \sigma_b\colon I^{\set{1,\dots,n-2,n-1}}\times I^{\set{n}}\to \scX_{n}\times F_{\set{n}}(R_b,d);
\]
\begin{definition}
For $i=n-1,n$
we denote by $t_i\in H_{n-1}(\scX_i,\scY_i)$ the image of the relative fundamental class of $I^{\set{1,\dots,n}\setminus\set{i}}$ along the product map $\sigma_{a_1}\times\dots\times\sigma_{a_{n-2}}\times\sigma_c$ (for $i=n$ we replace the last factor by $\bar\sigma_c$).

We denote by $s\in H_1(R_{d,b}),\del R_{d,b})$ the image of the relative fundamental class of $I$ along $\sigma_b$.
\end{definition}
For $i=n-1,n$, up to regarding $s$ as a class in $H_1(F_{\set{i}}(R_{d,b}),F_{\set{i}}(\del R_{d,b}))$,  the previous discussion shows that $x'_i=(-1)^{n-i}t_i\times s$, where the homology cross product is with respect to the product decomposition of the couple $(\scU'_i,\scV'_i)$ from Corollary \ref{cor:homeocouple}. The difference in sign comes from the above remark about possible change of orientation along canonical identifications of the two copies of $N'_1$ with the product of two cubes of dimensions $n-1$ and 1.

In a similar way, for $i=n-1,n$ we have inclusions $Q_i\subset \scU'_i$. For $i=n$, the product decomposition $W_{n-1}\times \md_b$ of $Q_n$ is compatible with the product decomposition $\scX_n\times F_{\set{n}}(R_{d,b})$ of $\scU'_n$ along the inclusion: the inclusion $Q_n\subset \scU'_n$ is the product of the inclusions $W_{n-1}\subset \scX_n$ and $\md_b\subset R_{d,b}\cong F_{\set{n}}(R_{d,b})$. In the same way we have a product inclusion of $Q_{n-1}$ in $\scU'_{n-1}$: more precisely, we have a canonical, orientation-reversing homeomorphism
$Q_{n-1}\cong (W_{n-2}\times\md_{\alpha\beta})\times \md_b$ and a canonical orientation-preserving homeomorphism $\scU'_{n-1}\cong \scX_{n-1}\times F_{\set{n-1}}(R_{d,b})$, and the $Q_{n-1}$ in $\scU'_{n-1}$ can be written as the product of the inclusions $(W_{n-2}\times\md_{\alpha\beta})\subset \scX_{n-1}$ and $\md_b\subset F_{\set{n-1}}(R_{d,b})$.

\begin{definition}
\label{defn:rn}
We denote by $s^{\vee}\in H^1(R_{d,b},\del R_{d,b})$ the cohomology class represented by the oriented  submanifold $\md_b$, which is disjoint from $\del R_{d,b}$. Again, we use that $\del R_{d,b}$ admits an open neighbourhood in $R_{d,b}$ that is disjoint from $\md_b$ and is homotopy equivalent to $\del R_{d,b}$.
\ref{subsec:list}.

We denote by $r_n\in H^{n-1}(\scX_n,\scY_n)$ the cohomology class represented by the oriented submanifold $W_{n-1}$, which is disjoint from $\scY_n$, and we denote by $r_{n-1}\in H^{n-1}(\scX_{n-1},\scY_{n-1})$ the cohomology class represented by the oriented submanifold $W_{n-2}\times\md_{\alpha\beta}$, which is disjoint from $\scY_{n-1}$.
\end{definition}
The notation for $s^{\vee}$ is justified by the fact that $\left<s,s^{\vee}\right>_{R_{d,b}\mathrm{rel.}\del R_{d,b}}=1$, using the convention from Subsection \ref{subsec:list}.

For $i=n-1,n$, up to regarding $s^{\vee}$ as a class in $H^1(F_{\set{i}}(R_{d,b}),F_{\set{i}}(\del R_{d,b}))$, the previous discussion shows that $y'_{i}=(-1)^{n-i}r_{i}\times s^{\vee}$,
where the cohomology cross product is with respect to the product decomposition of the couple $(\scU'_i,\scV'_i)$ from Corollary \ref{cor:homeocouple}. Again, the difference of sign is due to the possible change in orientation in the identification of $Q_i$ with the product of an oriented $(n-1)$-manifold and $\md_b$.

Using the compatibility of Kronecker pairing and cross product, together with the second part of Corollary \ref{cor:homeocouple}, we arrive at the following proposition, which complements Lemma \ref{lem:KrxpiPhiypi}
\begin{proposition}
\label{prop:Kroneckertiri}
The Kronecker pairing $\left<x,\phi^*y\right>$ can be computed as
\[
\left<t_{n-1},\Phi^*(r_{n-1})\right>_{\scX_{n-1}\mathrm{rel.}\scY_{n-1}}+\left<t_n,\Phi^*(r_n)\right>_{\scX_n\mathrm{rel.}\scY_n}.
\]
\end{proposition}
In fact, we will prove that the first summand vanishes, whereas the second summand is $1$. This is loosely speaking due to the fact that $r_{n-1}$ is represented by the manifold $W_{n-2}\times \md_{\alpha\beta}$, which splits as a product, whereas $r_n$ is represented by the manifold $W_{n-1}$, which does not split as a product.

We remark that at this point only the spaces $\scX_i$ and $\scY_i$ are involved in our computations, for $i=n-1,n$: these spaces arise as subspaces of $F_{\set{1,\dots,n}\setminus\set{i}}(\cF^R)$: thus our problem has been simplified, in that we are now considering ordered configurations of $n-1$ points (instead of $n$) on the subsurface $\cF^R\subset\msurf$ from Notation \ref{nota:FR}.

\subsection{Putting fog away from \texorpdfstring{$R_{A,\alpha}$}{Ra,alpha} and \texorpdfstring{$R_{c,\beta}$}{Rc,beta}}

\begin{notation}
\label{nota:me}
We denote by $e_\alpha, e_{\alpha\beta}\subset\bar\cF$ the images of $d_{\alpha}$ and $d_{\alpha\beta}$ along $\Phi^{-1}$, respectively, and by $\me_\alpha,\me_{\alpha\beta}$ their interior.
\end{notation}
Recall Definitions \ref{defn:Wi} and \ref{defn:rn}: the cohomology class $\Phi^*(r_n)\in H^{n-1}(\scX_n,\scY_n)$ is represented by the submanifold $\Phi^{-1}(W_{n-1})\subset\scX_n\setminus\scY_n$, containing all configurations $(z_1,\dots,z_{n-1})$ such that the points $z_1,\dots,z_{n-1}$ occur in this order on the oriented open segment $\me_\alpha$. Similarly, $\Phi^*(r_{n-1})\in H^{n-1}(\scX_{n-1},\scY_{n-1})$ is represented by the submanifold $\Phi^{-1}(W_{n-2}\times F_{\set{n}}(\md_{\alpha\beta}))$, containing configurations $(z_1,\dots,z_{n-2},z_n)$ such that $(z_1,\dots,z_{n-2})$ occur in this order on $\me_\alpha$, and $z_n\in\me_{\alpha\beta}$.

We apply another excision argument to compute the Kronecker pairings from Proposition \ref{prop:Kroneckertiri}.
First, we refine the notation from Subsection \ref{subsec:list}.
\begin{notation}
\label{nota:Raialpha}
We fix small, disjoint tubular neighbourhoods $U_{a_1},\dots,U_{a_{n-2}}\subset A$ of $a_1,\dots,a_{n-2}$; we denote $R_{a_i,\alpha}=U_{a_i}\cap R_{A,\alpha}$, which is homeomorphic to $[0,1]\times(0,1)$. We let the image of the relative fundamental class of $I$ along $\sigma_{a_i}$ be the preferred generator of $H_1(R_{a_i,\alpha},\del R_{a_i,\alpha})\cong\Z$, and similarly we let the image of the fundamental class of $I$ along $\sigma_c$ be the preferred generator of $H_1(R_{c,\beta},\del R_{c,\beta})$.
\end{notation}
Note that also $H_1(R_{A,\alpha},\del R_{A,\alpha})$ is canonically identified with $\Z$, since each inclusion $(R_{a_i,\alpha},\del R_{a_i,\alpha})\subset (R_{A,\alpha},\del R_{A,\alpha})$ is a homotopy equivalence (and induces the same choice of generator on first homology).
\begin{definition}
For $i=n-1,n$ we define $\scX'_i\subset\scX_i$ as the open subspace of configurations $(z_1,\dots,z_{n-2},z_{2n-1-i})$ such that $z_j\in R_{a_j,\alpha}$ for all $1\le j\le n-2$, and such that $z_{2n-1-i}\in R_{c,\beta}$. We let $\scY'_i=\scX'_i\cap \scY_i$.
\end{definition}
We have a homeomorphism of couples
\[
(\scX'_i,\scY'_i)\cong (R_{a_1,\alpha},\del R_{a_1,\alpha})\times\dots\times (R_{a_{n-2},\alpha},\del R_{a_{n-2},\alpha})\times (R_{c,\beta},\del R_{c,\beta}).
\]

Recall that, for $i=n-1,n$, the homology class $t_i\in H_{n-1}(\scX_i,\scY_i)$ is the image of the relative fundamental class of $I^{\set{1,\dots,n}\setminus\set{i}}$ along the map $\sigma_{a_1}\times\dots\times\sigma_{a_{n-2}}\times\sigma_c$ (for $i=n$, the last factor being $\bar\sigma_c$); the image of the latter product map is however contained in $\scX'_i$, and thus we can define the homology class $t'_i\in H_{n-1}(\scX'_i,\scY'_i)$ in the same way, so that $t'_i\mapsto t_i$ along the map induced in homology by the inclusion of couples $(\scX'_i,\scY'_i)\subset(\scX_i,\scY_i)$. We can thus define $r'_i\in H^{n-1}(\scX'_i,\scY'_i)$ as the restriction of $\Phi^*(r_i)$, and we have an equality
\[
\left<t_i,\Phi^{*}(r_i)\right>_{\scX_i\mathrm{rel.}\scY_i}=\left<t'_i,r'_i\right>_{\scX'_i\mathrm{rel.}\scY'_i}.
\]
Moreover, the previous discussion shows the equality $t'_i=\sigma_{a_1}[I,\del I]\times\dots\times\sigma_{a_{n-2}}[I,\del I]\times \sigma_c[I,\del I]$ (for $i=n$ replace $\sigma_c$ by $\bar\sigma_c$, which accounts on a change of sign); since $t'_i$ splits as a homology cross product of first homology classes, it is worth to study more closely the class $r'_i$, i.e. the intersection of the support of $r_i$ with $\scX'_i$.

\subsection{A combinatorial formula}
\begin{definition}
\label{defn:fjgj}
We let $f_1,\dots,f_\mu\subset \me_{\alpha}=\Phi^{-1}(\md_{\alpha})$ be the sequence of segments in which $\me_{\alpha}$ intersects $R_{A,\alpha}\cup R_{c,\beta}$: each segment inherits an orientation from $\me_{\alpha}$, and the segments are listed in the order in which they appear along $\me_{\alpha}$.
Similarly, we define $g_1,\dots,g_\nu\subset\me_{\alpha\beta}=\Phi^{-1}(\md_{\alpha\beta})$ as the sequence of segments in the intersection $\me_{\alpha\beta}\cap (R_{A,\alpha}\cup R_{c,\beta})$.

We define $\theta(f_i)\in\set{\ba,\bb}$ so that $f_i\subset R_{A,\alpha}$ if $\theta(f_i)=\ba$, and $f_i\subset R_{c,\beta}$ if $\theta(f_i)=\bb$; similarly we define $\theta(g_j)$.

For all $i$ such that $\theta(f_i)=\ba$, we define $\epsilon(f_i)\in\set{\pm1}$ so that $f_i$ represents the cohomology class $\epsilon(f_i)$ in $H^1(R_{A,\alpha},\del R_{A,\alpha})\cong \Z$, where the last identification is dual to the one from Notation \ref{nota:Raialpha}.
Similarly, for all $i$ such that $\theta(f_i)=\bb$, we define $\epsilon(f_i)\in\set{\pm1}$ by considering $H^1(R_{c,\beta},\del R_{c,\beta})\cong \Z$. And analogously, we define $\epsilon(g_j)\in\set{\pm1}$ for all $j$.
\end{definition}
We first focus on the case $i=n$: the intersection of $\Phi^{-1}(W_{n-1})$ with $\scX'_n$ splits as a disjoint union of products of segments: the disjoint union is parametrised by the set of all sequences $(i_1,\dots,i_{n-1})$ of indices $1\le i_j\le \mu$ satisfying the following:
\begin{itemize}
    \item $i_1\le \dots\le i_{n-1}$, and each equality $i_j=i_{j+1}$ implies $\epsilon(f_{i_j})=+1$;
    \item $\theta(f_{i_j})=\ba$ for $1\le j\le n-2$;
    \item $\theta(f_{i_{n-1}})=\bb$.
\end{itemize}
The product of segments corresponding to 
$(i_1,\dots,i_{n-1})$ is
\[
(f_{i_1}\cap R_{a_1,\alpha})\times\dots\times (f_{i_{n-2}}\cap R_{a_{n-2},\alpha})\times (f_{i_{n-1}}\cap R_{c,\beta}).
\]
The previous product, seen as a proper $(n-1)$-submanifold of $\scX'_n$ which is disjoint from $\scY'_n$, supports a cohomology class whose Kronecker pairing with $t'_n$ is
\begin{equation}
    \epsilon(f_{i_1})\cdot\dots\cdot\epsilon(f_{i_{n-2}})\cdot (-\epsilon(f_{i_{n-1}}))=\pm 1\in\Z,
\end{equation}
where the minus sign in the last factor is due to the fact that $t'_n$ is described in terms of $\bar\sigma_c$ rather than $\sigma_c$. Summing over all choices of $(i_1,\dots,i_{n-1})$, we can in principle compute $\left<t'_n,r'_n\right>$.

Let us now consider the case $i=n-1$: the intersection of $\Phi^{-1}(W_{n-2}\times\md_{\alpha\beta})$ with $\scX'_{n-1}$ splits as a disjoint union of products of segments: the disjoint union is parametrised by the set of all sequences $(i_1,\dots,i_{n-2};l)$ of indices $1\le i_j\le \mu$ and $1\le l\le \nu$ satisfying the following:
\begin{itemize}
    \item $i_1\le \dots\le i_{n-2}$, and each equality $i_j=i_{j+1}$ implies $\epsilon(f_{i_j})=+1$;
    \item $\theta(f_{i_j})=\ba$ for $1\le j\le n-2$;
    \item $\theta(g_l)=\bb$.
\end{itemize}
The product of segments corresponding to 
$(i_1,\dots,i_{n-2};l)$ is
\[
(f_{i_1}\cap R_{a_1,\alpha})\times\dots\times (f_{i_{n-2}}\cap R_{a_{n-2},\alpha})\times (g_l\cap R_{c,\beta}).
\]
The previous product, seen as a proper $(n-1)$-submanifold of $\scX'_{n-1}$ which is disjoint from $\scY'_{n-1}$, supports a cohomology class whose Kronecker pairing with $t'_{n-1}$ is
\[
\epsilon(f_{i_1})\cdot\dots\cdot\epsilon(f_{i_{n-2}})\cdot \epsilon(g_l)=\pm 1\in\Z.
\]
Summing over all choices of $(i_1,\dots,i_{n-2};l)$, we can in principle compute $\left<t'_{n-1},r'_{n-1}\right>$. Note however that the conditions on $(i_1,\dots,i_{n-2};l)$ split as a list of conditions only on $(i_1,\dots,i_{n-2})$ and one condition only on $l$, and the sum computing $\left<t'_{n-1},r'_{n-1}\right>$ is thus equal to the product of the following two sums:
\begin{itemize}
    \item the sum $\sum_{(i_1,\dots,i_{n-2})}\epsilon(f_{i_1})\cdot\dots\cdot\epsilon(f_{i_{n-2}})$, for $(i_1,\dots,i_{n-2})$ satisfying the conditions in the previous list;
    \item the sum
    $\sum_{l\colon \theta(g_l)=\beta}\epsilon (g_l)$.
\end{itemize}

\subsection{Reinterpretation in terms of contents}

It will be helpful for computations to record the above findings in the context of words in free monoids and groups.

\begin{definition}
We denote by $\Mon\langle \ba^{\pm1}, \bb^{\pm1}\rangle$ the free monoid in the letters $\ba$, $\ba^{-1}$, $\bb$ and $\bb^{-1}$; it contains unreduced words $\uw=(w_1^{\epsilon_1},w_2^{\epsilon_2},\dots, w_h^{\epsilon_h})$, where $h\ge0$, and for all $1\le i\le h$ we have $w_i\in\set{\ba,\bb}$ and $\epsilon_i\in\set{\pm1}$; composition is given by concatenation. We denote by $\bF\langle\ba,\bb\rangle$ the free group generated by the letters $\ba,\bb$. We denote by $\iota\colon \Mon\langle \ba^{\pm1},\bb^{\pm1}\rangle \to\bF\langle\ba,\bb\rangle$ the surjective monoid homomorphism sending $(w^{\epsilon})\mapsto w^{\epsilon}$.
\end{definition}

\begin{definition}[Contents]
\label{defn:contents}
Fix a word 
 $\uw=(w_1^{\epsilon_1},w_2^{\epsilon_2},\dots, w_h^{\epsilon_h})\in \Mon\langle \ba, \bb\rangle$. For $k\ge 0$, an $\ba^k\bb$\textit{-occurence in $\uw$} is a sequence $s=(i_1,\dots,i_{k+1})$ of indices $1\le i_j\le h$ satisfying
 \begin{enumerate}
     \item $i_1\le\dots\le i_{k+1}$ and equality $i_j=i_{j+1}$ implies $\epsilon_{i_j}=+1$;
     \item $w_{i_j}=\ba$ for $1\le j\le k$;
     \item $w_{i_{k+1}}=\bb$.
 \end{enumerate}
The \textit{sign of $s$} is the product $\sigma(s)=\prod_{1\le j\le k+1} \epsilon_{i_j}$ and the $\ba^k\bb$\textit{-content of $\uw$} is the signed count 
\[
\cont(\ba^k\bb;\uw):=\sum_{\ba^k\bb \text{-occurences } s \text{ in }\uw} \sigma(s).
\]

Analogously, for $k\ge 0$, an $\ba^k$\textit{-occurence in $\uw$} is a sequence $s=(i_1,\dots,i_k)$ satisfying conditions $(1)$ and $(2)$ above. The \textit{sign} $\sigma(s)$ of an occurrence $s$, and the $\ba^k$\textit{-content of $\uw$}, denoted $\cont(\ba^k,\uw)$, are analogously defined.
\end{definition}
Observe $\cont(\ba^0,\uw)=1$ for all words $\uw$, since there exists precisely one occurrence $s=()$ of $\ba^0$ in $\uw$, carrying sign $1$.
In the light of Definition \ref{defn:contents}, the previous discussion translates into the following proposition.
\begin{proposition}\label{prop:reductiontocontents}
We have equalities
\[
\langle t_n',r_n'\rangle =-\cont\big(\ba^{n-2}\bb\ ,\ (\theta(f_1)^{\epsilon(f_1)},\dots,\theta(f_\mu)^{\epsilon(f_\mu)})\big);
\]
\[
\langle t_{n-1}',r_{n-1}'\rangle=\cont\big(\ba^{n-2}\ ,\ (\theta(f_1)^{\epsilon(f_1)},\dots,\theta(f_\mu)^{\epsilon(f_\mu)})\big)\cdot \big(\theta(g_1)^{\epsilon(g_1)},\dots,\theta(g_\nu)^{\epsilon(g_\nu)}\big)_{\bb},
\]
where $(-)_{\bb}:\Mon\langle \ba^{\pm1},\bb^{\pm1}\rangle \to \Z$ is the
composition of $\iota\colon \Mon\langle \ba^{\pm1},\bb^{\pm1}\rangle\to \bF\langle \ba^{\pm1},\bb^{\pm1}\rangle$ and the group homomorphism $\bF\langle \ba^{\pm1},\bb^{\pm1}\rangle\to\Z$ sending
$\alpha\mapsto 0$ and $\beta\mapsto 1$. 
\end{proposition}

In the next section we will use the previous results to prove that $\langle t_n',r_n'\rangle$ is equal to $1$, whereas $\langle t_{n-1}',r_{n-1}'\rangle$ vanishes because $\big(\theta(g_1)^{\epsilon(g_1)},\dots,\theta(g_\nu)^{\epsilon(g_\nu)}\big)_{\bb}$ vanishes.
\section{After Fog}\label{sec:afterfog}

In this section we finally compute 
$\langle t'_i ,r'_i\rangle$ for $i=n-1,n$, by computing the contents from Proposition \ref{prop:reductiontocontents}.
\begin{definition}
\label{defn:cFplus}
Recall Subsection \ref{subsec:list} and Notation \ref{nota:FR}. We consider the subspace $\cF^R\cup (d\cap \bar\cF) \cup\delout\cF\subset\msurf$, and define
\[
\cF_+:=\pa{\cF^R\cup (d\cap \bar\cF) \cup\delout\cF}/\delout\cF
\]
as the quotient of the space $\cF^R\cup (d\cap \bar\cF) \cup\delout\cF$ by $\delout\cF$. See Figure \ref{fig:cFplus}.

We denote by $\Phi_+\colon\cF_+\to\cF_+$ the homeomorphism induced on the quotient by the restriction of $\Phi$ on $\cF^R\cup (d\cap \bar\cF) \cup\delout\cF$.
\end{definition}
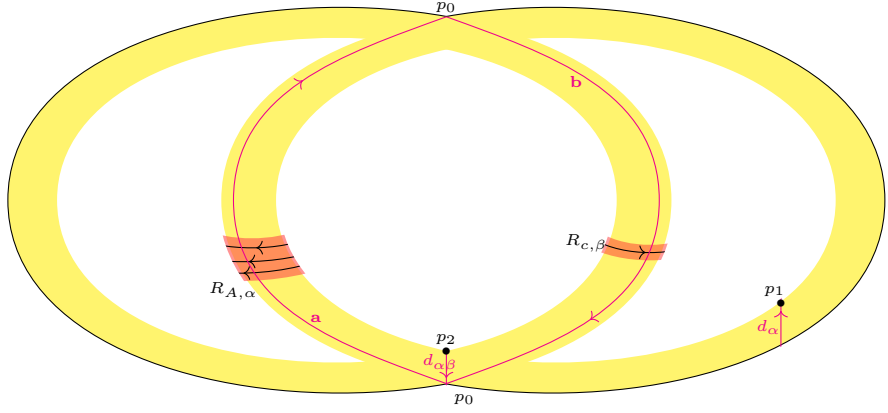
\begin{figure}
  \centering  
  \begin{tikzpicture}[xscale=.8,yscale=.8, decoration={markings, mark=at position 0.74 with {\arrow{>}}}]

\begin{scope}
\clip (0.2,3.3) rectangle (7.5,-3.3);
\draw[looseness=1,thick](.3,0) to[out=90, in=90] ++(10.9,0) to[out=-90,in=-90] ++(-10.9,0);
\end{scope}
\begin{scope}
\clip (14.8,3.3) rectangle (7.5,-3.3);
\draw[looseness=1,thick](14.7,0) to[out=90, in=90] ++(-10.9,0) to[out=-90,in=-90] ++(10.9,0);
\end{scope}

\fill[looseness=1, yellow!70!white] (.3,0) to[out=90, in=90] ++(10.9,0) to[out=-90,in=-90] ++(-10.9,0)    to[out=0,in=180] ++(.8,0)  to[out=-90, in=-90] ++(9.2,0) to[out=90,in=90] ++(-9.2,0)    to[out=180,in=0] ++(-.8,0);
\begin{scope}[shift={(15,0)},xscale=-1]
    \fill[looseness=1, yellow!70!white] (.3,0) to[out=90, in=90] ++(10.9,0) to[out=-90,in=-90] ++(-10.9,0)    to[out=0,in=180] ++(.8,0)  to[out=-90, in=-90] ++(9.2,0) to[out=90,in=90] ++(-9.2,0)    to[out=180,in=0] ++(-.8,0);
\end{scope}

\begin{scope}
\clip[looseness=1] (14.75,0) to[out=-90, in=-90] ++(-11.0,0) to[out=0,in=180] ++(1.0,0)  to[out=-90, in=-90] ++(9.1,0) to[out=0,in=180] ++(.9,0);
\fill[red!70!white, opacity=.6, looseness=.9] (1.5,0) to[out=90, in=90] ++(5.1,0) to[out=-90,in=-90] ++(-5.1,0) node[left]{\tiny$A$} to[out=0,in=180] ++(1.6,0) to[out=-90,in=-90] ++(2.4,0) to[out=90,in=90] ++(-2.4,0) to [out=180,in=0] ++(-1.6,0); 
\draw[looseness=.9, postaction={decorate}] (1.75,0) node[below]{\tiny$a_1$}  to[out=90, in=90]  ++(4.6,0)  to[out=-90,in=-90] ++(-4.6,0);
\draw[looseness=.85, postaction={decorate}] (2.1,0) node[below]{\tiny$a_2$} node[right]{...} to[out=90, in=90]  ++(4.1,0) to[out=-90,in=-90] ++(-4.1,0);
\draw[looseness=.9, postaction={decorate}] (2.8,0) node[below]{\tiny$a_{n\!-\!2}$}  to[out=90, in=90]  ++(3.0,0) to[out=-90,in=-90] ++(-3.0,0);
\draw[looseness=1, magenta, postaction={decorate}] (13,-2.5) to[out=90, in=-90] (13,-1.6);
\draw[looseness=1, magenta, postaction={decorate}] (7.5,-1.4) to[out=-90, in=90] (7.5,-3.5);
\end{scope}

\begin{scope}
\clip[looseness=1] (.25,0) to[out=-90, in=-90] ++(11.0,0) to[out=180,in=0] ++(-1.0,0) to[out=-90,in=-90] ++(-9.1,0) to[out=180,in=0] ++(-.9,0);
 \fill[looseness=1.1, red!80!white, opacity=.5] (9.3,0) to[out=90, in=90] ++(3.1,0) to[out=-90,in=-90] ++(-3.1,0) to[out=0,in=180] ++(.4,0) to[out=-90,in=-90] ++(2.3,0) to[out=90,in=90] ++(-2.3,0) to[out=180,in=0] ++(-.4,0);
\draw[looseness=1.1, postaction={decorate}] (12.2,0)  node[left]{\tiny $c$}  to[out=90, in=90] ++(-2.7,0) to[out=-90,in=-90] ++(2.7,0);
\end{scope}

\node at (4,-1.5) {\tiny$R_{A,\alpha}$};
\node at (9.8,-.7) {\tiny$R_{c,\beta}$};
\node[magenta] at (12.8,-2.1) {\tiny$d_\alpha$};
\node[magenta] at (7.4,-2.7) {\tiny$d_{\alpha\beta}$};

\node at (7.8,-3.3) {\tiny$p_0$}; \node at (7.5,3.2) {\tiny$p_0$};
\node at (13,-1.7) {\tiny$\bullet$}; \node at (12.9,-1.5) {\tiny$p_1$};
\node at (7.5,-2.5) {\tiny$\bullet$}; \node at (7.5,-2.3) {\tiny$p_2$};

\draw[looseness=1, magenta, postaction={decorate}] (7.5,-3.04) to[out=160, in=-90] node[right]{\tiny $\ba$} (4,0) to[out=90, in=-160] (7.5,3.04);
\draw[looseness=1, magenta, postaction={decorate}] (7.5,3.04) to[out=-20, in=90] node[left]{\tiny $\bb$} (11,0) to[out=-90, in=20] (7.5,-3.04);
  \end{tikzpicture}
  \caption{The surface $\cF_+$.}
    \label{fig:cFplus}
\end{figure}
We notice that the quotient $\bar\cF/\delout\cF$ is a compact surface of genus 0 with 3 boundary curves; the space $\cF_+$ is obtained from the latter by removing parts of the boundary, leaving only the following portions:
\begin{itemize}
    \item $\del R_{A,\alpha}$ and $\del R_{c,\beta}$;
    \item the unique point in $\del d_{\alpha}\setminus\delout\cF$, which we call $p_2$.
    \item the unique point in $\del d_{\alpha\beta}\setminus\delout\cF$, which we call $p_1$.
\end{itemize}
We moreover call $p_0\in \cF_+$ the quotient point of $\delout\cF$.
We consider the fundamental groupoid $\Pi_1(\cF_+;p_0,p_1,p_2)$, i.e. the category whose three objects are $p_0,p_1,p_2$ and whose morphisms from $p_i$ to $p_j$ are the homotopy classes of paths from $p_i$ to $p_j$.
In particular, all morphisms of this groupoid are generated by the following morphisms (and their inverses):
\begin{itemize}
    \item the paths $d_\alpha\colon p_0\to p_1$ and $d_{\alpha\beta}\colon p_2\to p_0$;
    \item the loops $\ba,\bb\colon p_0\to p_0$ from Figure \ref{fig:cFplus}: $\ba$ is contained in the closure of $U_\alpha$ in $\cF_+$ and intersects all segments $a_i\cap R_{A,\alpha}$ once, transversely and from right;
    $\bb$ is contained in the closure of $U_\beta$ in $\cF_+$ and intersects the segment $c\cap R_{c,\beta}$ once, transversely and from right.
\end{itemize}
We compose morphisms in $\Pi_1(\cF_+;p_0,p_1,p_2)$ according to the convention for which the following holds:
$\ba d_\alpha$ is a defined morphism $p_0\to p_1$, whereas the composition $d_\alpha\ba$ is not defined.

Note that $\pi_1(\cF_+,p_0)\cong \bF\langle\ba,\bb\rangle$ is free on the generators $\ba$ and $\bb$. Moreover, the set of morphisms in $\Pi_1(\cF_+;p_0,p_1,p_2)$ from $p_0$ to $p_1$ is $\bF\langle\ba,\bb\rangle d_\alpha$, and the set of morphisms from $p_2$ to $p_0$ is $d_{\alpha\beta}\bF\langle\ba,\bb\rangle$.

Since the homeomorphism $\Phi_+\colon\cF_+\to\cF_+$ fixes the points $p_0,p_1,p_2$, it acts on $\Pi_1(\cF_+;p_0,p_1,p_2)$. The images of $d_\alpha$ and $d_{\alpha\beta}$ along $\Phi_+^{-1}$ are the classes of the paths $e_{\alpha}\colon p_0\to p_1$ and $e_{\alpha\beta}\colon p_2\to p_0$.
By Definition \ref{defn:fjgj} we have equalities of morphisms in $\Pi_1(\cF_+;p_0,p_1,p_2)$
\[
e_{\alpha}=\theta(f_1)^{\epsilon(f_1)}\dots\theta(f_\mu)^{\epsilon(f_\mu)} d_\alpha \colon p_0\to p_1;
\]
\[
e_{\alpha\beta}=d_{\alpha\beta} \theta(g_1)^{\epsilon(g_1)}\dots\theta(g_\nu)^{\epsilon(g_\nu)} \colon p_2\to p_0.
\]
We remind that, in the previous expression, the words $\theta(f_1)^{\epsilon(f_1)}\dots\theta(f_\mu)^{\epsilon(f_\mu)}$ and $\theta(g_1)^{\epsilon(g_1)}\dots\theta(g_\nu)^{\epsilon(g_\nu)}$, representing elements in $\bF\langle\ba,\bb \rangle$, may not be reduced.

In the following we compute $e_{\alpha}$ and $e_{\alpha\beta}$ as morphisms in $\Pi_1(\cF_+;p_0,p_1,p_2)$, using that $\Phi^*$ is isotopic, as a homeomorphism of $\cF_+$, to the iterated commutator of Dehn twists
$[D_{\alpha},[D_{\alpha},[\dots[D_{\alpha},D_{\beta}]]]\dots]$, where $D_{\alpha}$ and $D_\beta$ are the Dehn twists along the curves $\alpha$ and $\beta$ in $\cF_+$, and $D_{\alpha}$ is repeated $n-2$ times.



A direct computation gives the following lemma.
\begin{lemma}\label{lem:actionofDehntwists}
The action of the Dehn twists $D_{\alpha}$, $D_{\beta}$ on $\Pi_1(\cF_+;p_0,p_1,p_2)$ is given on the generating morphisms by
  \begin{align*}
         D_{\alpha}*\ba &=\ba, & D_{\beta}*\ba &=\bb\ba\bb^{-1},\\
         D_{\alpha}*\bb&=\ba\bb\ba^{-1}, &  D_{\beta}*\bb &=\bb,\\
         D_{\alpha}*d_{\alpha}&=\ba d_{\alpha}, &  D_{\beta}*d_{\alpha}&=d_{\alpha},\\
         D_{\alpha}*d_{\alpha\beta}&=d_{\alpha\beta}\ba^{-1}, &  D_{\beta}*d_{\alpha\beta}&=d_{\alpha\beta}\bb^{-1}.
     \end{align*}
In particular $D_{\alpha}$ and $D_{\beta}$ act on $\pi_1(\cF_+;p_0)=\langle \ba, \bb\rangle$ by the conjugations $\uw\mapsto \ba \uw\ba^{-1}$ and $\uw\mapsto \bb \uw\bb^{-1}$ respectively.
\end{lemma}
\begin{proof}
Let $\gamma\subset\cF_+$ be an immersed arc in $\cF_+$ representing a morphism $q\to q'$ in $\Pi_1(\cF_+;p_0,p_1,p_2)$, i.e. $q,q'\in\set{p_0,p_1,p_2}$, and let $\delta\subset\cF_+$ be either $\alpha$ or $\beta$; assume that $\gamma$ and $\delta$ are transverse. To compute the image of the morphism $\gamma\colon q\to q'$ along $D_{\delta}$, we orient $\gamma$ from $q$ to $q'$, and list the intersection points in $\gamma\cap \delta $ as $\{q_1,\dots,q_h\}$, so that, using the orientation of $\gamma$, we have $q<q_1<\dots<q_h<q'$. We then have that $D_{\delta}*\gamma\colon q\to q'$ is homotopic to the concatenation
$$D_{\delta}*\gamma\simeq\gamma':= \gamma|_{[q,q_1]}\delta_{q_1}\gamma|_{[q_1,q_2]}\delta_{q_2}\cdot\dots\cdot\delta_{q_h}\gamma|_{[q_h,q']}.$$
Here $\gamma|_{[\hat q\tilde q]}$ is the segment of $\gamma$ from $\hat q$ to $\tilde q$, and $\delta_{q_i}$ is the simple loop $\delta$ based at $q_i$, with orientation so that when arriving from $\gamma$ to $\delta$ we turn left.

We draw two \textit{barriers} $b_{\ba}$ and $b_{\bb}$ that are arcs with endpoints in $\partial \cF_+$ so that
\begin{itemize}
    \item[(i)] $ \cF_+-(b_{\ba}\cup b_{\bb})$ is simply connected;
    \item[(ii)] $b_{\ba}$ intersects $\ba$ transversely once, and is disjoint from $\bb$;
    \item[(iii)] $b_{\bb}$ intersects $\bb$ transversely once, and is disjoint from $\ba$. 
\end{itemize}

We first assume $q=q'=p_0$, i.e. that $\gamma$ is a loop based at $p_0$, and we further assume that $\gamma$ is transverse to the barriers and that no intersection point of $\gamma$ and $\delta$ lies on the barriers. As a consequence, also $\gamma'$ is transverse to the barriers: here, since $\gamma'$ may not be embedded, we consider it as a parametrised loop $[0,1]\to\cF_+$.
We may present the morphism $\gamma'\colon p_0\to p_0$ as a word in $\ba^{\pm1}$ and $\bb^{\pm1}$ with the following recipe. 
For each $\bd=\ba,\bb$, we say that an intersection time between $\gamma'$ and $b_{\bd}$ is an element $t\in[0,1]$ such that $\gamma'(t)\in b_{\bd}$. We say that
$t$ is positive if $\gamma'$ intersects $b_{\bd}$ in the same direction as $\bd$, and negative otherwise. We record the intersection times of $\gamma'$ with the barriers, taking them in the order that corresponds to the orientation of $\gamma'$; we obtain a sequence $w_1,\dots,w_k\in \{\ba,\bb\}$, capturing the labels of the barrier on which each intersection point lies, and a sequence of signs $\epsilon_1,\dots,\epsilon_k\in\{\pm 1\}$ corresponding to the sign of the intersections. We then have an equality
\[
\gamma'=w_1^{\epsilon_1}\dots w_k^{\epsilon_{k}}\in \pi_1(\cF_+;p_0).
\]
If $\gamma\colon q\to q'$ is not a loop, then we concatenate $\gamma'\colon q\to q'$ with a morphism $q'\to q$ (typically $d_{\alpha}^{-1}$ or $d_{\alpha\beta}^{-1}$) that avoids the barriers in order to obtain a loop based at $p_0$, perform the same reasoning and deconcatenate in the end.

In Figure \ref{fig:Dehntwist}, we exhibit by the above method that $D_{\beta}*\ba=\bb\ba\bb^{-1}$. The other cases are analogous.
\begin{figure}[h]
  \centering 
   \begin{tikzpicture}[xscale=.8,yscale=.8, decoration={markings, mark=at position 0.74 with {\arrow{>}}}]

\begin{scope}
\clip (0.2,3.3) rectangle (7.5,-3.3);
\draw[looseness=1,thick](.3,0) to[out=90, in=90] ++(10.9,0) to[out=-90,in=-90] ++(-10.9,0);
\end{scope}
\begin{scope}
\clip (14.8,3.3) rectangle (7.5,-3.3);
\draw[looseness=1,thick](14.7,0) to[out=90, in=90] ++(-10.9,0) to[out=-90,in=-90] ++(10.9,0);
\end{scope}

\fill[looseness=1, yellow!70!white] (0.3,0) to[out=90, in=90] ++(10.9,0) to[out=-90,in=-90] ++(-10.9,0)    to[out=0,in=180] ++(1.5,0)  to[out=-90, in=-90] ++(7.5,0) to[out=90,in=90] ++(-7.5,0)    to[out=180,in=0] ++(-1.5,0);
\begin{scope}[shift={(15,0)},xscale=-1]
    \fill[looseness=1, yellow!70!white] (0.3,0) to[out=90, in=90] ++(10.9,0) to[out=-90,in=-90] ++(-10.9,0)    to[out=0,in=180] ++(1.5,0)  to[out=-90, in=-90] ++(7.5,0) to[out=90,in=90] ++(-7.5,0)    to[out=180,in=0] ++(-1.5,0);
\end{scope}

\draw[looseness=1,blue,thick, dashed] (1,0)  to[out=90,in=90] (10.4,0) to[out=-90,in=-90] (1,0);
\draw[blue] (1.2,0) node[right]{$\beta$};

\draw[looseness=1, thick, postaction={decorate}] (7.5,-3.04) to[out=160, in=-90]  (4,0) node[right]{$\ba$} to[out=90, in=-160] (7.5,3.04);

\node at (7.8,-3.3) {$p_0$}; \node at (7.5,3.3) {$p_0$};

\filldraw[black] (6.6,-2.7) circle (2pt) node[above]{$q_1$};
\filldraw[black] (6.6,2.7) circle (2pt) node[above]{$q_2$};
\end{tikzpicture}
  
  \begin{tikzpicture}[xscale=.8,yscale=.8, decoration={markings, mark=at position 0.74 with {\arrow{>}}}]

\begin{scope}
\clip (0.2,3.3) rectangle (7.5,-3.3);
\draw[looseness=1,thick](.3,0) to[out=90, in=90] ++(10.9,0) to[out=-90,in=-90] ++(-10.9,0);
\end{scope}
\begin{scope}
\clip (14.8,3.3) rectangle (7.5,-3.3);
\draw[looseness=1,thick](14.7,0) to[out=90, in=90] ++(-10.9,0) to[out=-90,in=-90] ++(10.9,0);
\end{scope}

\fill[looseness=1, yellow!70!white] (0.3,0) to[out=90, in=90] ++(10.9,0) to[out=-90,in=-90] ++(-10.9,0)    to[out=0,in=180] ++(1.5,0)  to[out=-90, in=-90] ++(7.5,0) to[out=90,in=90] ++(-7.5,0)    to[out=180,in=0] ++(-1.5,0);
\draw[thick, red] (11.2,0) node[right]{$b_{\bb}$} to (9.3,0);
\draw[postaction={decorate},red] (11,0.2) to (11,-0.4);
\draw (10.9,-0.4);

\begin{scope}[shift={(15,0)},xscale=-1]
    \fill[looseness=1, yellow!70!white] (0.3,0) to[out=90, in=90] ++(10.9,0) to[out=-90,in=-90] ++(-10.9,0)    to[out=0,in=180] ++(1.5,0)  to[out=-90, in=-90] ++(7.5,0) to[out=90,in=90] ++(-7.5,0)    to[out=180,in=0] ++(-1.5,0);
    \draw[thick, red] (11.2,0) to (9.3,0) node[right]{$b_{\ba}$};
    \draw[postaction={decorate},red] (9.5,-0.2) to  (9.5,0.4);
\end{scope}

\draw[looseness=1,blue, thin, dashed] (1,0)  to[out=90,in=90] (10.4,0) to[out=-90,in=-90] (1,0);


\draw[looseness=1,thick]
(7.5,-3.04) to[out=160, in=5] ++(-0.6,0.15) to[out=185,in=-90] (0.65,0.1)
to[out=90,in=90] (10.05,0)
to[out=-90,in=-25] (7.5-1.8,-3.04+1.1)
to[out=155,in=-90] (4.4,0)
to[out=90,in=-155] (7.5-1.8,3.04-1.1)
to[out=25,in=0] ++(.2,0.5)
to[out=180,in=90] (1.3,0)
to[out=-90,in=-90] (10.75,0)
to[out=90,in=-20] (7.3,2.9)
to[out=160,in=-160] (7.5,3.04);

\draw (0.5,2.7) node[right]{$D_{\beta}*\ba$};

\draw[looseness=1,,thick,postaction={decorate}]
(10.05,0) to[out=-90,in=70] (9.95,-0.6) node[left]{\tiny$1$};
\draw[looseness=1,,thick,postaction={decorate}]
(10.75,0) to[out=90,in=-74] (10.67,0.6) node[right]{\tiny$3$};
\draw[looseness=1,,thick,postaction={decorate}]
(4.4,0) to[out=90,in=-108] (4.48,0.6) node[right]{\tiny$2$};

\node at (7.8,-3.3) {$p_0$}; \node at (7.5,3.3) {$p_0$};

  \end{tikzpicture}
  \caption{Action of the Dehn twist along $\beta$ (blue curve) on $\ba$ (black curve) . Since the curves intersect twice at $q_1$ and $q_2$, the loop $D_{\beta}*\ba$ (thick black curve, below) is homotopic to the surgery of $\ba$ with two copies of $\beta$. It intersects the (blue horizontal) barriers three times (labelled $1,2,3$), in the order $b_{\bb}$, $b_{\ba}$, $b_{\bb}$ with respective signs $+,+,-$, so $D_{\beta}*\ba=\bb\ba\bb^{-1}$.}
  \label{fig:Dehntwist}
\end{figure}
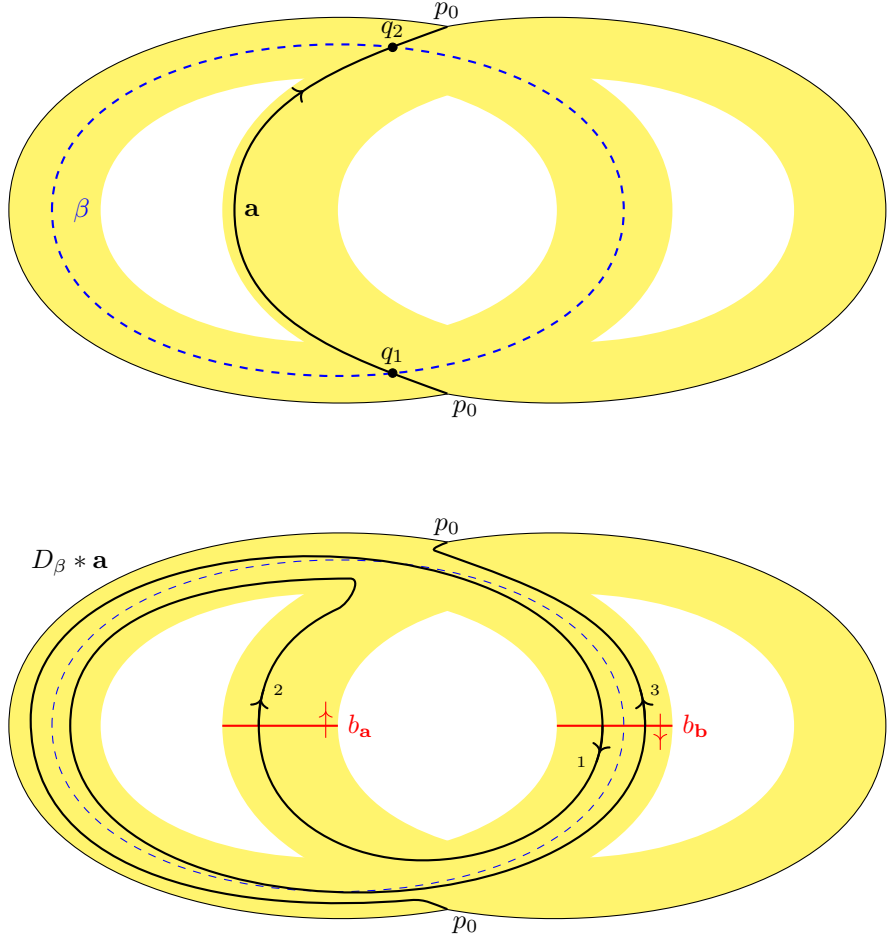
\end{proof}

\begin{corollary}\label{cor:formulasealphabeta}
The morphisms $e_{\alpha}$ and $e_{\alpha\beta}$ in $\Pi_1(\cF_+;p_0,p_1,p_2)$ are given by the following formulas
$$e_{\alpha}=[\ba,\dots[\ba,\bb]\dots]d_{\alpha};\quad\quad
e_{\alpha\beta}=d_{\alpha\beta}([\ba,\dots[\ba,\bb]\dots])^{-1},$$
where $\ba$ is repeated $n-2$ times.
\end{corollary}
\begin{proof}
Let $W(-,\star)$ be a word in the letters $-$ and $\star$, and consider $W(-,\star)$ as a rule to produce an element of a group by plugging into $-$ and $\star$ elements of that group. 
By Lemma \ref{lem:actionofDehntwists}, we have the equalities
\[
\begin{split}
D_{\alpha}*(W(\ba,\bb)d_{\alpha})&=\ba W(\ba,\bb)\ba^{-1}\ba d_{\alpha}=\ba W(\ba,\bb)d_{\alpha};\\
D_{\beta}*(W(\ba,\bb)d_{\alpha})&=\bb W(\ba,\bb)\bb^{-1}d_{\alpha}.
\end{split}
\]
Applying the above two rules we obtain, for another generic word $W'(-,\star)$, the equality
$$W'(D_{\alpha},D_{\beta})*(W(\ba,\bb)d_{\alpha})=W'(\ba,\bb)W(\ba,\bb)(W'(\ba,1))^{-1}d_{\alpha}.$$
If the word $W'$ is a commutator of two other words, $W'(\ba,1)=1$ in $\pi_1(\cF_+,p_0)$. Now recall that $\Phi^{-1}_+=[D_{\alpha},\dots[D_{\alpha},D_{\beta}]\dots]$ with $D_{\alpha}$ appearing $n-2$ times: we can apply the previous formula with the iterated commutator $W'(-,\star)=[-,[-,\dots[-,\star]\dots]]$, where $-$ occurs $n-2$ times, together with the trivial word $W(-,\star)=1$, to obtain
$$e_{\alpha}=\Phi^{-1}_+*d_{\alpha}=[D_{\alpha},\dots[D_{\alpha},D_{\beta}]\dots]*d_{\alpha}=[\ba,\dots[\ba,\bb]\dots]d_{\alpha},$$
where $\gamma_\alpha$ occurs $n-2$ times, as desired.

For $d_{\alpha\beta}$ a similar process yields
$$D_{\delta}*(d_{\alpha\beta}W(\ba,\bb))=d_{\alpha\beta}\gamma_{\delta}^{-1}\gamma_{\delta}W(\ba,\bb)\gamma_{\delta}^{-1}=d_{\alpha\beta}W(\ba,\bb)\gamma_{\delta}^{-1}$$
for $\delta=\alpha,\beta$, and consequently
$$W'(D_{\alpha},D_{\beta})*(d_{\alpha\beta}W(\ba,\bb))=d_{\alpha\beta}W(\gamma_\alpha,\gamma_\beta)(W'(\ba,\bb))^{-1},$$
from which the desired result again follows.
\end{proof}

It follows that the word $\theta(g_1)^{\epsilon(g_1)}\dots\theta(g_\nu)^{\epsilon(g_\nu)}$ is equal in $\pi_1(\cF_+,p_0)$ to the commutator $d_{\alpha\beta}^{-1}e_{\alpha\beta}=[\ba,[\ba,\dots[\ba,\bb]\dots]]^{-1}$, which, by virtue of being a commutator, lies in the kernel of the homomorphism $[-]_{\bb}$, that has codomain the abelian group $\Z$. We deduce
\begin{proposition}
$\langle t_{n-1}',r_{n-1}'\rangle=0$.
\end{proposition}

\subsection{The non-vanishing content}
It remains to compute
$$\langle t'_n,r'_n\rangle=-\cont\big(\ba^{n-2}\bb\ ,\ (\theta(f_1)^{\epsilon(f_1)},\dots,\theta(f_\mu)^{\epsilon(f_\mu)})\big)=-\cont\big(\ba^{n-2}\bb\ ,\ [\ba,\dots[\ba,\bb]\dots]\big),$$
where $\alpha$ appears $n-2$ times.
To do so effectively, we introduce the following algebra.

\begin{definition}
The \textit{non-commutative} algebra $R$ is defined by starting with the quotient
$$\Z\langle \bx,\by\rangle/(\by\bx, \by^2)$$
of the free associative algebra $\Z\langle \bx,\by\rangle$ by the two-sided ideal generated by $\by\bx$ and $\by^2$, and by taking the completion with respect to the two-sided ideal generated by $\bx$ and $\by$. In other words,
$$R=\{\bp(\bx)+\bq(\bx)\by \ :\  \bp(\bx),\bq(\bx)\text{ are formal power series in } \bx\},$$
and multiplication satisfies
$$(\bp(\bx)+\bq(\bx)\by)(\bp'(\bx)+\bq'(\bx)\by)=\bp(\bx)\bp'(\bx)+\big(\bp(\bx)\bq'(\bx)+\bq(\bx)\bp'(0)\big)\by,$$
where $\bp'(0)$ is the constant term of $\bp'(\bx)$.
\end{definition}

Write $R^{\ge 1}=\{\bp(\bx)+\bq(\bx)\by\in R\ :\  \bp(0)=0\}$.
\begin{lemma}\label{lem:inverses}
In the multiplicative monoid $1+R^{\ge 1}$,  elements of the form 
$$1+\bx\bp(\bx)\text{ and } 1+\bp(\bx)\by$$
have unique two-sided inverses
$$\frac{1}{1+\bx\bp(\bx)}=\sum_{k\ge 0}(-\bx\bp(\bx))^k\text{ and } 1-\bp(\bx)\by,$$
respectively.
\end{lemma}
\begin{proof}
Two-sided inverses are unique in a monoid as long as they exist, so it suffices to provide some two-sided inverses for $1+\bx\bp(\bx)$ and $1+\bp(\bx)\by$. We have 
\[
(1+\bp(\bx)\by)(1-\bp(\bx)\by)=1=(1-\bp(\bx)\by)(1+\bp(\bx)\by),
\]
and for $1+\bx\bp(\bx)$ we merely observe that each coefficient of a power of $\bx$ in the sum $\sum_{k\ge 0} (-\bx\bp(\bx))^{k}$ can be computed by a finite sum.
\end{proof}

\begin{definition}
The total content of $\uw\in \Mon\langle \ba^{\pm1},  \bb^{\pm1}\rangle$ is  defined as
$$\Cont(\uw):=\sum_{k\ge 0} \cont(\ba^k, \uw)\bx^k+\cont(\ba^k\bb,\uw)\bx^k\by\in 1+R^{\ge 1}\subset R.$$
\end{definition}
It is a straightforward calculation that $\Cont(\ba^{-1})=1+\bx$ and $\Cont(\bb)=1+\by$.

\begin{lemma}
The total content is multiplicative, in the sense that, for $\uw,\uw'\in \Mon\langle \ba^{\pm1},  \bb^{\pm1}\rangle$, we have \begin{equation}\label{eq:totalcontentmult}
    \Cont(\uw\uw')=\Cont(\uw)\Cont(\uw').
\end{equation}
\end{lemma}
\begin{proof}
Write $\uw=(w_1^{\epsilon_1},\dots,w_h^{\epsilon_h})$ and $\uw'=(w_{h+1}^{\epsilon_{h+1}},\dots, w_{h+h'}^{\epsilon_{h+h'}})$. For every $\ba^k\bb$-occurence $s=(i_1,\dots,i_{k+1})$ in $\uw\uw'$, there is an index $0\le l\le k+1$ s.t. $i_j\le h$ for $j\le l$ and $i_j\ge h+1$ for $j\ge l+1$. If $l=k+1$, then $s$ is an $\ba^k\bb$-occurence in $\uw$. If instead $j<k+1$, then $s_1:=(i_1,\dots,i_l)$ is an $\ba^l$-occurence in $\uw$ and $s_2:=(i_{l+1},\dots,i_{k+1})$ is an $\ba^{k-l}\bb$-occurence in $\uw'$. Conversely for any $0\le j\le k+1$, any $\ba^l$-occurence $s_1$ in $\uw$ and any $\ba^{k-l}\bb$-occurence $s_2$ in $\uw'$, the concatenation $(s_1,s_2)$ is an $\ba^k\bb$-occurence in $\uw\uw'$. We have thus provided a bijection between $\ba^k\bb$-occurences in $\uw\uw'$ and the disjoint union of $\ba^k\bb$-occurences in $\uw$ with the set of triples consisting of an index $0\le l\le k$, an $\ba^l$-occurence $s_1$ in $\uw$ and $\ba^{k-l}\bb$-occurence $s_2$ in $\uw'$. Moreover, in the latter case, the sign is multiplicative,
i.e. $\sigma(s_1s_2)=\sigma(s)$. Therefore we have
\begin{equation}\label{eq:cnmult}
    \cont(\ba^k\bb,\uw\uw')=\cont(\ba^k\bb,\uw)+\sum_{0\le  l\le k}\cont(\ba^l,\uw)\cont(\ba^{k-l}\bb,\uw').
\end{equation}
A similar argument on $\ba^k$-occurences shows that 
\begin{equation}\label{eq:cnalphamult}
    \cont(\ba^k,\uw\uw')=\sum_{0\le  l\le k}\cont(\ba^l\uw)\cont(\ba^{k-l},\uw').
\end{equation}

Finally, we can $\Cont(\uw\uw')$ and $\Cont(\uw)\Cont(\uw')$ using the multiplication rules of $R$: the equality $C(ww')=C(w)C(w')$ reduces to the equalities \eqref{eq:cnmult} and \eqref{eq:cnalphamult} for all $k\ge 0$.
\end{proof}

\begin{proposition}\label{prop:contentinvariance}
The contents $\cont(\ba^k\bb,-)$, $\cont(\ba^k,-)$ and $\Cont$ are well-defined as functions from the free group $\bF\langle \ba,\bb\rangle$, i.e. they factor along the surjection $\iota\colon\Mon\langle\ba^{\pm1},\bb^{\pm1} \rangle\to\bF\langle\ba,\bb\rangle$.
\end{proposition}
\begin{proof}
The word $(\ba,\ba^{-1})\in \Mon\langle\ba^{\pm1},\bb^{\pm1} \rangle$ has no $\ba^k\bb$-occurences for any $k\ge 0$, and it has precisely two $\ba^k$-occurences of opposite signs for all $k\ge 1$; therefore $\Cont(\ba,\ba^{-1})=1$ which is equal to $\Cont()$, the total content of the empty word. Similarly, $\Cont(\uw)=1$ for each of the words $\uw=(\ba^{-1},\ba),(\bb,\bb^{-1})$ and $(\bb^{-1},\bb)$ in $\Mon\langle\ba^{\pm1},\bb^{\pm1} \rangle$. The multiplicativity of $\Cont$ implies that cancelling a pair of opposite letters in an unreduced word of $\Mon\langle\ba^{\pm1},\bb^{\pm1} \rangle$ does not change the total content; thus $\Cont$ is a well defined function on $\bF\langle \ba,\bb\rangle$. By reading the coefficients of $\Cont$, the same holds for  $\cont(\ba^k\bb,-)$ and $\cont(\ba^k,-)$, for all $k\ge0$.
\end{proof}
Proposition \ref{prop:contentinvariance} will be used in the following lemma.
\begin{lemma}\label{lem:commutingwitha}
Let $\uw\in\bF\langle\ba,\bb\rangle$. If $\Cont(\uw)=1+\bp(\bx)\by$, then 
$$\Cont([\ba,\uw])=1+\frac{\bx}{1-\bx}\bp(\bx)\by.$$
\end{lemma}
\begin{proof}
The word $\ba$ has no $\ba^k\bb$-occurence and precisely one $\ba^k$-occurence for every $k\ge 0$, carrying positive sign; thus $\Cont(\ba)=\sum_{k\ge0}\bx^k=\frac{1}{1-\bx}$.
By the multiplicativity of the total content, $\Cont(\uw^{-1})$ and $\Cont(\ba^{-1})$ are two-sided inverses of $\Cont(\uw)$ and $\Cont(\ba)$ in the monoid $1+R^{\ge1}$; thus by Lemma \ref{lem:inverses} we have $\Cont(\ba^{-1})=1-\bx$ and $\Cont(\uw^{-1})=1-\bp(\bx)\by$.
We can then compute
\begin{align*}
    \Cont([\ba,\uw])&=\Cont(\ba \uw \ba^{-1} \uw^{-1})
    =\Cont(\ba)\Cont(\uw)\Cont(\ba^{-1})\Cont(\uw^{-1})\\
    &=\frac{1}{1-\bx}(1+\bp(\bx)\by)(1-\bx)(1-\bp(\bx)\by)\\
    &=\frac{1}{1-\bx}(1+\bp(\bx)\by-\bx)(1-\bp(\bx)\by)\\
    &=\frac{1}{1-\bx}(1-\bx+\bx\bp(\bx)\by)\\
    &=1+\frac{\bx}{1-\bx}\bp(\bx)\by.
\end{align*}
\end{proof}
Putting together Proposition \ref{prop:reductiontocontents}, Corollary \ref{cor:formulasealphabeta}, Proposition \ref{prop:contentinvariance} and Lemma \ref{lem:commutingwitha}, we obtain the following corollary.
\begin{corollary}
$\langle t_n',r_n'\rangle=-1$.
\end{corollary}
\begin{proof}
Repeated applications of Lemma \ref{lem:commutingwitha} show that 
$$\Cont([\ba,[\ba,\dots[\ba,\bb]\dots]])=1+\Big(\frac{\bx}{1-\bx}\Big)^{n-2}\by,$$
where $\ba$ appears $n-2$ times in the iterated commutator. Modulo $\bx^{n-1}\by$, this expression is equal to $1+\bx^{n-2}\by$, so that
$\cont(\ba^{n-2}\bb,[\alpha,[\alpha,\dots[\alpha,\beta]\dots]])=1$. By Corollary \ref{cor:formulasealphabeta} we have
\[
[\ba,[\ba,\dots[\ba,\bb]\dots]]=\theta(f_1)^{\epsilon(f_1)}\cdot\dots\cdot\theta(f_\mu)^{\epsilon(f_\mu)}\in\pi_1(\cF_+,p_0)\cong\bF\langle\ba,\bb\rangle,
\]
and by Proposition \ref{prop:contentinvariance} we obtain the equality
\[
\cont(\ba^{n-2}\bb,[\ba,[\ba,\dots[\ba,\bb]\dots]])=\cont(\ba^{n-2}\bb,\theta(f_1)^{\epsilon(f_1)}\cdot\dots\cdot\theta(f_\mu)^{\epsilon(f_\mu)}).
\]
Finally, Proposition \ref{prop:reductiontocontents} gives the equality
\[
\langle t_n',r_n'\rangle=-\cont(\ba^{n-2}\bb,\theta(f_1)^{\epsilon(f_1)}\cdot\dots\cdot\theta(f_\mu)^{\epsilon(f_\mu)}).
\]
\end{proof}



\section{Configuration spaces of closed surfaces}
In this section we study configuration spaces of closed surfaces, and prove Theorem \ref{thm:main2}.
\subsection{A list of subspaces of \texorpdfstring{$\Sigma_{g+1}$}{Sg+1}}
Let $g\ge2$ and regard $\Sigma_{g+1}$ as the closed surface obtained from $\surf=\Sigma_{g,1}$ by glueing the surface $\cT:=\Sigma_{1,1}$ along the boundary. We denote by $\mcT$ the interior of $\cT$, which is an open subset of $\Sigma_{g+1}$.
Let $l\subset\mcT$ be an oriented simple closed curve; we fix an oriented, simple closed curve $\hat d\subset\Sigma_{g+1}$ intersecting $l$ once, transversely and positively, and such that $\hat d\cap \surf=d$. See Figure \ref{fig:arcsandcurvesclosed}.
\begin{figure}
  \centering
  \begin{tikzpicture}[xscale=.8,yscale=.8, decoration={markings,mark=at position 0.1 with {\arrow{>}}}]
  
\fill[looseness=1, yellow!70!white] (.3,0) to[out=90, in=90] ++(10.9,0) to[out=-90,in=-90] ++(-10.9,0)    to[out=0,in=180] ++(.3,0)  to[out=-90, in=-90] ++(10.2,0) to[out=90,in=90] ++(-10.2,0)    to[out=180,in=0] ++(-.3,0);
\begin{scope}[shift={(15,0)},xscale=-1]
    \fill[looseness=1, yellow!70!white] (.3,0) to[out=90, in=90] ++(10.9,0) to[out=-90,in=-90] ++(-10.9,0)    to[out=0,in=180] ++(.3,0)  to[out=-90, in=-90] ++(10.2,0) to[out=90,in=90] ++(-10.2,0)    to[out=180,in=0] ++(-.3,0);
\end{scope}

    \fill[looseness=1.1, red!80!white, opacity=.5] (11.3,-.2) to[out=90, in=90] ++(3.05,0) to[out=-90,in=-90] ++(-3.05,0) to[out=0, in=180] ++(.4,0) to[out=-90,in=-90] ++(2.3,0) to[out=90,in=90] ++(-2.3,0) to[out=180,in=0] ++(-.4,0); 

    \draw[looseness=1.1, postaction={decorate}] (11.5,-.2) to[out=90, in=90] ++(2.7,0)  node[left]{\tiny $b$};

    \draw[looseness=1, thick,blue] (.5,0) to[out=90, in=90] node[below]{\tiny$\beta$} ++(10.5,0)  to[out=-90,in=-90] ++(-10.5,0);
    \draw[looseness=.7, thick, blue] (10.4,.2) to[out=50, in=60] (9,1.1) node[left]{\tiny$\beta'$}; 
    
\begin{scope}[shift={(15,0)},xscale=-1]
    \draw[looseness=1, thick, blue] (.5,0) to[out=90, in=90] node[below]{\tiny$\alpha$} ++(10.5,0) to[out=-90,in=-90]  ++(-10.5,0);
    \draw[looseness=.7, thick, blue] (10.4,.2) to[out=50, in=60] (9,1.1) node[right]{\tiny$\alpha'$}; 
\end{scope}

    \fill[looseness=1, opacity=.8, white] (8,0) to[out=40,in=180] ++(3,2.5) to[out=0,in=140] ++(3,-2.5) to[out=180, in=0] (11.8,0) to[out=130,in=50] (10.2,0) to[out=180, in=0] (8,0);
\begin{scope}[shift={(15,0)},xscale=-1]
    \fill[looseness=1, opacity=.8, white] (8,0) to[out=40,in=180] ++(3,2.5) to[out=0,in=140] ++(3,-2.5) to[out=180, in=0] (11.8,0) to[out=130,in=50] (10.2,0) to[out=180, in=0] (8,0);
\end{scope}

    \draw[looseness=.7, thick, blue] (9,1.1) to[out=-120,in=-130] (10.4,.2);
\begin{scope}[shift={(15,0)},xscale=-1]
    \draw[looseness=.7, thick, blue] (9,1.1) to[out=-120,in=-130] (10.4,.2);
\end{scope}

    \fill[red!70!white, opacity=.6, looseness=.9] (1.5,0) to[out=90, in=90] ++(5.1,0) to[out=-90,in=-90] ++(-5.1,0) node[left]{\tiny$A$} to[out=0,in=180] ++(1.6,0) to[out=-90,in=-90] ++(2.4,0) to[out=90,in=90] ++(-2.4,0) to [out=180,in=0] ++(-1.6,0);
    \fill[looseness=1.1, red!80!white, opacity=.5] (9.3,0) to[out=90, in=90] ++(3.1,0) to[out=-90,in=-90] ++(-3.1,0) to[out=0,in=180] ++(.4,0) to[out=-90,in=-90] ++(2.3,0) to[out=90,in=90] ++(-2.3,0) to[out=180,in=0] ++(-.4,0);
    \fill[looseness=2, red!40!white] (11.1,-.9) to[out=90,in=90] ++(1,0) to[out=-90,in=-90] ++(-1,0);

  
    \draw[looseness=1, thick] (1,0) to[out=40,in=180] ++(3,2.5) to[out=0,in=140] ++(3,-2.5);
    \draw[looseness=1, thick] (3.2,0) to[out=50,in=130] ++(1.6,0);
    
    \draw[looseness=1, thick] (8,0) to[out=40,in=180] ++(3,2.5) to[out=0,in=140] ++(3,-2.5);
    \draw[looseness=1, thick] (10.2,0) to[out=50,in=130] ++(1.6,0);

    \draw[looseness=.9, postaction={decorate}] (1.75,0) node[below]{\tiny$a_1$}  to[out=90, in=90]  ++(4.6,0)  to[out=-90,in=-90] ++(-4.6,0);
    \draw[looseness=.85, postaction={decorate}] (2.1,0) node[below]{\tiny$a_2$} node[right]{...} to[out=90, in=90]  ++(4.1,0) to[out=-90,in=-90] ++(-4.1,0);
    \draw[looseness=.9, postaction={decorate}] (2.8,0) node[below]{\tiny$a_{n\!-\!2}$}  to[out=90, in=90]  ++(3.0,0) to[out=-90,in=-90] ++(-3.0,0);

    \draw[looseness=1.1, postaction={decorate}] (9.5,0)  node[left]{\tiny $c$}  to[out=-90, in=-90] ++(2.7,0) to[out=90,in=90] ++(-2.7,0);

    \fill[looseness=1, orange, opacity=.4] (7.2,-3.5) to[out=90, in=-90] (8.5,0) to [out=90, in=90] (13.2,0) to[out=-90, in=90] (13.2,-2.38) to [out=-90,in=0] (10,-6.7) to[out=180,in=-90] (7.2,-3.5) to[out=10,in=-190] (7.6,-3.51) to[out=-90,in=180] (10,-6.3) to[out=0,in=-90]    
    (12.8,-2.55) to[out=90,in=-90] (12.8,0) to[out=90, in=90] (8.9,0) to[out=-90, in=90] (7.6,-3.51) to [out=-190, in=10] (7.2,-3.5);
    
    \draw[looseness=1.1,] (11.5,-.2) to[out=-90, in=-90] ++(2.7,0);

    \draw[looseness=1, magenta, postaction={decorate}] (13,-2.5) to[out=90, in=-90] (13,0) to[out=90, in=90] (8.7,0) to[out=-90, in=90] node[left]{\tiny$\hat d$} (7.43,-3.5) to[out=-90,in=180] (10,-6.5) to[out=0,in=-90] (13,-2.5); 

    \node at (11.7,-.7) {\tiny$\bullet$}; \node at (11.6,-.9) {\tiny$P$};
    \fill[opacity=.4, looseness=2] (11.2,-.9) to[out=90,in=90] ++(.8,0) to[out=-90,in=-90] ++(-.8,0);
    
    \draw[looseness=.8,thick] (0,0) to[out=-90,in=-90] ++(15,0) to[out=90,in=90] ++(-15,0);
    
\node at (4,-1.5) {\tiny$R_{A,\alpha}$};
\node at (10.3,-1) {\tiny$R_{c,\beta}$};
\node at (12.6,-2) {\tiny$R_{d,\alpha}$};
\node at (8,-3.2) {\tiny$R_{d,\alpha\beta}$};
\node at (13.3,-0.8) {\tiny$R_{d,b}$};

    \draw[looseness=.8] (-.2,0) to[out=90,in=90] (15.2,0);
    \draw[looseness=1.2] (-.2,0) to[out=-90,in=180] (7.5,-7.1) to[out=0,in=-90] (15.2,0);

\begin{scope}[shift={(3.5,-4.4)},yscale=-1]
\begin{scope}[shift={(15,0)},xscale=-1]
    \fill[looseness=1, opacity=.8, white] (8,0) to[out=40,in=180] ++(3,2.5) to[out=0,in=140] ++(3,-2.5) to[out=180, in=0] (11.8,0) to[out=130,in=50] (10.2,0) to[out=180, in=0] (8,0);
\end{scope}
    \draw[looseness=1, thick] (1,0) to[out=40,in=180] ++(3,2.5) to[out=0,in=140] ++(3,-2.5);
    \draw[looseness=1, thick] (3.2,0) to[out=50,in=130] ++(1.6,0);
\end{scope}

    \fill[looseness=1.1, red!80!white, opacity=.5] (5.8,-5) to[out=90, in=90] ++(3.4,0) to[out=-90,in=-90] ++(-3.4,0) to[out=0,in=180] ++(.4,0) to[out=-90,in=-90] ++(2.6,0) to[out=90,in=90] ++(-2.6,0) to[out=180,in=0] ++(-.4,0);
    \draw[looseness=1.1, postaction={decorate}] (6,-5)  node[left]{\tiny $l$}  to[out=90, in=90] ++(3,0) to[out=-90,in=-90] ++(-3,0);

    \node at (7,-3.8) {\tiny$R_{\hat d,l}$};
    \node at (2,-1.5) {$\surf$};
    \node at (2,-4) {$\cT$};

  \end{tikzpicture}
  \caption{The surface $\Sigma_{g+1}$ and some relevant subspaces of it, analogous to Figure \ref{fig:arcsandcurves}.}
    \label{fig:arcsandcurvesclosed}
\end{figure}
We also fix small tubular neighbourhoods $U_l$ and $U_{\hat d}$ of $l,\hat d\subset\Sigma_{g+1}$, respectively, and let $R_{d,l}$ be the intersection $ U_l\cap \bar U_{\hat d}$.

The open inclusion $\msurf\sqcup\mcT\subset\Sigma_{g+1}$ gives rise to an open inclusion $F_n(\msurf)\times F_{\set{n+1}}(\mcT)\subset F_{n+1}(\Sigma_{g+1})$ (see Notation \ref{nota:FSZ}). We note that $F_{n+1}(\Sigma_{g+1})$ is an oriented $2n+2$-manifold, and thus we can define cohomology classes of $F_{n+1}(\Sigma_{g+1})$ by giving a properly embedded, oriented submanifold. 
\subsection{The homology, cohomology and mapping classes in the closed case}
\begin{definition}
We let $\hat x\in H_{n+1}(F_{n+1}(\Sigma_{g+1}))$ be the image of the cross product homology class $x\times [l]$, where $x\in H_n(F_n(\surf))$ is the homology class from Subsection \ref{subsec:x}, and $[l]\in H_1(F_1(\mcT))$ is the fundamental homology class of $l$.
\end{definition}
\begin{definition}
We let $\hat N_2\subset F_{n+1}(\Sigma_{g+1})$ be the subspace of configurations $(z_1,\dots,z_{n+1})$ such that all $z_i$ lie on $\hat d$, and up to cyclic permutations they occur in this order along $\hat d$. We note that $\hat N_2$ is an oriented $n+1$-submanifold of $F_{n+1}(\Sigma_{g+1})$, as it admits a proper parametrisation by $S^1\times\Delta^n$.
We let $\hat y\in H^{n+1}(F_{n+1}(\Sigma_{g+1}))$ be the cohomology class represented by $\hat N_2$.
\end{definition}
\begin{definition}
We let $\hat\Phi\colon\Sigma_{g+1}\to\Sigma_{g+1}$ be the diffeomorphism obtained by extending $\Phi$ with the identity of $\cT$. We let $\hat\phi\in \Gamma_{g+1}$ be the corresponding mapping class.
\end{definition}
Note that $\hat\phi\in J_{g+1}(n)$. We want to prove that $\hat\phi$ acts non-trivially on $\hat x$, and we will again prove that the Kronecker pairing $\langle \hat\phi_*(\hat x)-\hat x,\hat y\rangle$ is non-zero. We notice that $\hat x$ is supported on configurations $(z_1,\dots,z_{n+1})$ such that $z_1\in A$, and in particular all such configurations do not belong to $\hat N_2$: it follows that $\langle \hat x,\hat y\rangle=0$.

\subsection{Putting fog in the closed case}
\begin{definition}
We let $\hat\scU\subset F_{n+1}(\Sigma_{g+1})$ be the open subspace $F_n(\msurf)\times F_{\set{n+1}}(U_l)$, i.e. the subspace of configurations $(z_1,\dots,z_{n+1})$ satisfying $z_1,\dots,z_n\in\msurf$ and  $z_{n+1}\in U_l$.
We let $\hat\scV\subset\scU$ be the closed subspace of configurations satisfying in addition that $z_{n+1}\not\in\mR_{\hat d,l}$.
\end{definition}
We note that both subspaces $\hat\scV\subset\hat\scU\subset F_{n+1}(\Sigma_{g+1})$ are invariant under the action of $\hat\Phi$; moreover $\hat x$ is the image of a homology class $\hat x'=x\times [l]\in H_{n+1}(\scU)$, so that we can compute $\langle \hat\phi_*(\hat x),\hat y\rangle_{F_{n+1}(\Sigma_{g+1})}$ by restricting $\hat y$ to the space $\hat\scU$, i.e. as $\langle \hat x',\hat y|_{\hat\scU}\rangle_{\hat\scU}$.

We then note that the intersection of $\hat N_2$ with $\scU=F_n(\msurf)\times F_{\set{n+1}}(U_l)$ is equal to the product $N_2\times (\hat d\cap U_l)$, and is in particular disjoint from $\hat\scV$. We can therefore denote $\hat y'\in H^{n+1}(\hat\scU,\hat\scV)$ be the relative cohomology class supported by $\hat N_2\cap \hat\scU$,
project $\hat x'$ to a relative homology class in $H_{n+1}(\hat\scU,\hat\scV)$, that by abuse of notation we still call $\hat x'$, and compute $\langle \hat x',\hat y|_{\hat\scU}\rangle_{\hat\scU}$ as $\langle \hat x',\hat y'\rangle_{\hat\scU\mathrm{rel.}\hat\scV}$.

Using the product decomposition of the couple $(\hat\scU,\hat\scV)=F_n(\msurf)\times(U_l,U_l\setminus \mR_{\hat d,l})$, we can compute the previous Kronecker product as $\langle x,y\rangle_{F_n(\msurf)}\cdot\langle [l],\hat d\cap U_l\rangle_{U_l\mathrm{rel.}U_l\setminus \mR_{\hat d,l}}$. And in the previous product the second factors is equal to $1$, whereas the first factor is equal to $-1$ by Theorem \ref{thm:main}. This completes the proof of Theorem \ref{thm:main2}.

\bibliographystyle{amsalpha}

\bibliography{biblio}

\end{document}